\documentclass[a4paper,10pt]{amsart}

\usepackage[latin1]{inputenc}   
\usepackage{amssymb,bbm,amsmath,amsthm}
\usepackage{latexsym}

\usepackage{graphicx}

\usepackage{curves} 
\usepackage{epic} 
\usepackage{eepic}
\usepackage{xcolor}

\usepackage{hyperref}
\usepackage{verbatim}

\input xy
\xyoption{all}

\newcommand{\R}{\ensuremath{\mathbb{R}}}
\newcommand{\Z}{\ensuremath{\mathbb{Z}}}
\newcommand{\C}{\ensuremath{\mathbb{C}}}
\newcommand{\Q}{\ensuremath{\mathbb{Q}}}
\newcommand{\N}{\ensuremath{\mathbb{N}}}
\newcommand{\F}{\ensuremath{\mathbb{F}}}

\newcommand{\Proj}{\ensuremath{\mathbb{P}}}

\newcommand{\lra}{\ensuremath{\longrightarrow}}

\newcommand{\J}{\ensuremath{\mathcal{J}}}
\newcommand{\calo}{\ensuremath{\mathcal{O}}}
\newcommand{\A}{\ensuremath{\mathbb{A}}} 
\newcommand{\x}{\ensuremath{\mathbf{x}}}

\DeclareMathOperator{\Spec}{Spec}
\DeclareMathOperator{\ord}{ord}

\DeclareMathOperator{\Jac}{Jac}
\DeclareMathOperator{\Div}{div}
\DeclareMathOperator{\Hom}{Hom}
\DeclareMathOperator{\End}{End}

\newcommand{\mc}[1]{\ensuremath{\mathcal{#1}}}   
\newcommand{\mf}[1]{\ensuremath{\mathfrak{#1}}}   

\newtheorem{Thm}{Theorem}[section]

\newtheorem{Cor}[Thm]{Corollary}

\newtheorem{Lem}[Thm]{Lemma}

\newtheorem{Prop}[Thm]{Proposition}

{\theoremstyle{definition}
       \newtheorem{defi}[Thm]{Definition}
       
	\newtheorem{Exercise}[Thm]{Exercise}
       \newtheorem{fact}[Thm]{Fact}	
       \newtheorem{Rmk}[Thm]{Remark}
       \newtheorem{ex}[Thm]{Example}

\numberwithin{equation}{Thm}

\setcounter{section}{0}
\setcounter{tocdepth}{2}

\title[Measuring Singularities with Frobenius]{Measuring Singularities with Frobenius: The Basics}
\author{Ang\'elica Benito, Eleonore Faber, Karen E. Smith}
 \thanks{The first author is partially supported by MTM2009-07291, by the Fellowship Fundación Ramón Areces para Estudios Postdoctorales 2011/12, and by Spanish Government in the frame of Postdoctoral mobility abroad Fellowship (code: EX-2010-0128).
The second author has been supported by a For Women in Science award 2011 of L'Or{\'e}al Austria, the Austrian commission for UNESCO and the Austrian Academy of Sciences and by
the Austrian Science Fund (FWF)
in frame of project P21461. The third author is partially supported by NSF grant DMS-1001764.}

\address[Ang\' elica Benito]{Department of Mathematics
 University of Michigan\\
 Ann Arbor, MI 48109}
 \email{abenitos@umich.edu}
\address[Eleonore Faber]{Fakult\"at f\"ur Mathematik\\
Universit\"at Wien, Austria}
\email[Eleonore Faber]{eleonore.faber@univie.ac.at}
\address[Karen E. Smith]{Department of Mathematics
 University of Michigan\\
 Ann Arbor, MI 48109}
\email{kesmith@umich.edu}

\begin{document}

\maketitle

\section{Introduction}
Consider a polynomial $f$ over some field $k$, vanishing at some point $x$ in $k^n$.  By definition, $f$ is smooth at $x$ (or the hypersurface defined by $f$ is smooth at $x$)   if and only if some partial derivative $\frac{\partial f}{\partial x_i}$ is non-zero there. Otherwise, $f$  is singular at $x$. But how singular?  Can we quantify the singularity of $f$ at $x$?

The  multiplicity  is perhaps the most naive measurement of singularities. Because $f$ is singular at $x$ if all the first order partial derivatives of $f$ vanish there, it is natural to say that $f$ is even more singular if also all the second order partials vanish, and so forth.  The {\it order,\/} or {\it multiplicity,\/}  of the singularity  at $x$ is the largest $d$ such that for all differential operators  $\partial$ of order  less than $ d$, $\partial f$ vanishes at $x$. Choosing coordinates so that $x$ is the origin,  it is easy to see that the multiplicity is simply the degree of the lowest degree term of $f$.

The multiplicity is an important first step in measuring singularities, but it is too crude to give a good measurement of singularities. For example, the polynomials $xy$ and $y^2 - x^3$ both define  singularities of multiplicity two, though the former is clearly less singular than the latter. 
%
%
%
Indeed, $xy$ defines a simple normal crossing divisor, whereas the singularity of the cuspidal curve defined by  $y^2 - x^3$ is quite complicated, and that of, for example $y^2 - x^{17}$ is even more so.  

\begin{center}
\setlength{\unitlength}{0.8mm}
\begin{picture}(147,50)
\allinethickness{1.1pt}

\put(20,27){\spline(-18,0)(18,0)}
\put(20,27){\spline(0,-18)(0,18)}
\put(11,13){\makebox(0,0){\footnotesize$\displaystyle yx=0$}}

\put(70,27){\spline(0,0)(1,0.5)(2.25,1.6875)(4,4)(6.25,7.8125)(9,13.5)(10.89,17.9685)}
\put(70,27){\spline(0,0)(1,-0.5)(2.25,-1.6875)(4,-4)(6.25,-7.8125)(9,-13.5)(10.89,-17.9685)}

\put(69,13){\makebox(0,0){\footnotesize$\displaystyle y^2=x^3$}}

\put(126,27){\spline(0,0)(1,0.01)(2.25,0.170859375)(4,1.28)(6.25,6.103515625)(7.29,10.4603532)(8.41,17.24987631)(8.7025,19.44258973)(9,21.87)}
\put(126,27){\spline(0,0)(1,-0.01)(2.25,-0.170859375)(4,-1.28)(6.25,-6.103515625)(7.29,-10.4603532)(8.41,-17.24987631)(8.7025,-19.44258973)(9,-21.87)}
\put(124,13){\makebox(0,0){\footnotesize$\displaystyle y^2=x^{17}$}}
\end{picture}

\vspace{-.3cm}

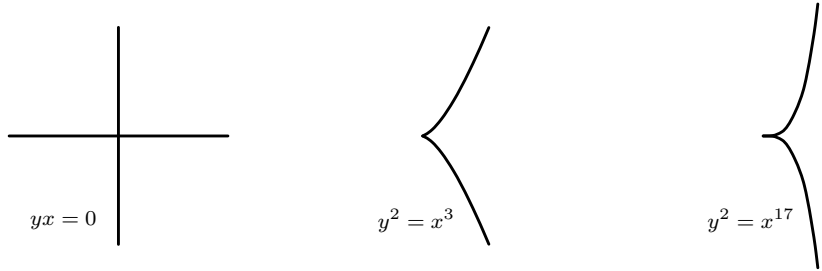
Figure 1. Curves with multiplicity $2$ at the origin.
\end{center}

\vspace{0.3cm}

%
%
%

This  paper describes the first steps toward understanding  a much more subtle measure of singularities 
 which arises naturally in three different contexts--- analytic, algebro-geometric, and finally, algebraic.  
Miraculously,  all three approaches  lead to essentially the same measurement of singularities:  the log canonical threshold (in characteristic zero) and the closely related $F$-pure threshold (in  characteristic $p$).  The log canonical threshold, or   complex singularity exponent,  can be defined analytically (via integration) or algebro-geometrically (via resolution of singularities). As such, it is defined only for polynomials over $\mathbb C$ or other characteristic zero fields. The $F$-pure threshold, whose name we shorten to F-threshold here, by contrast, is defined only in prime characteristic. Its definition makes use of the Frobenius, or $p$-th power map. Remarkably, these two completely different ways of quantifying singularities turn out to be intimately related. As we will describe, if we fix a polynomial with integer coefficients, the $F$-threshold of its ``reduction mod $p$" approaches its log canonical threshold as $p$ goes to infinity.

Both the log canonical threshold and the $F$-threshold  can be interpreted as critical numbers for the behavior of certain associated ideals, called the multiplier ideals in the characteristic zero  setting, and the test ideals in the characteristic $p$ world. Both  naturally give rise to higher order analogs, called ``jumping numbers." We will also introduce these refinements.

We present only the first steps in understanding these invariants, with an emphasis on the prime characteristic setting. Attempting only to demystify the concepts in the simplest cases, 
we make no effort to discuss the most general case, or to describe the many interesting connections with deep ideas 
in analysis, topology, algebraic geometry, number theory, and commutative algebra. The reader who is inspired to dig deeper will find plenty of more sophisticated  survey articles
and  a plethora of connections to many ideas throughout mathematics, including  the Bernstein-Sato polynomial \cite{Ko}, Varchenko's work on mixed  Hodge structures on the vanishing cycle \cite{Var}, the Hodge spectrum \cite{St},  the Igusa Zeta Function \cite{Ig}, motivic integration and jet schemes \cite{Mu},  Lelong numbers \cite{De3},  Tian's invariant for studying K\"ahler-Einstein metrics \cite{De2},  various vanishing theorems for cohomology \cite[Chapter 9]{Laz}, birational rigidity \cite{dEM}, Shokurov's approach to termination of flips \cite{Shok},  Hasse invariants \cite{BS},  the monodromy action on the Milnor fiber, Frobenius splitting and tight closure. 
  
 There are several surveys which are both more  sophisticated and pay more attention to history. 
 In particular,  the classic survey by Koll\'ar in \cite[Sections 8--10]{Ko} contains a deeper discussion of the characteristic zero theory, as do the more recent lectures of Budur \cite{Bu2}, mainly from the algebro-geometrical perspective.  For a more analytic discussion, papers of  Demailly are worth looking at, such as the article \cite{De1}.  Likewise, for the full characteristic $p$ story, Schwede and Tucker's survey of test ideals \cite{ST} is very nice.  The survey  \cite{Mu2} contains a modern account of  both the characteristic $p$ and characteristic zero theory. 

{\it Acknowledgements: } The authors are grateful to Daniel Hernandez, Jeffrey Lagarias, Kai Rajala and Kevin Tucker for many useful comments. In addition, two referees made excellent suggestions and spotted many misstatements. These notes began as a lecture notes from a  mini-course by the third author at  a workshop at Luminy entitled ``Multiplier Ideals in Commutative Algebra and Singularity 
Theory'' organized by Michel Granger, Daniel Naie, Anne Pichon in January 2011; the authors are grateful to the organizers of that conference. They were then polished at the workshop ``Computational F-singularities" organized by Karl Schwede, Kevin Tucker and Wenliang Zhang in Ann Arbor Michigan in May 2012.

 \section{Characteristic zero: Log canonical threshold and Multiplier ideals}

%
%
%

In this section we work with polynomials over the complex numbers $\C$. Let $ \C^N \overset{f}\longrightarrow \C$ be a polynomial (or analytic)  function, vanishing at a point $x$. 

\subsection{Analytic approach}
Approaching singularities from an analytic  point of view, we  consider how fast the (almost everywhere defined) function 
$$\xymatrix@R=0pc@C=2pc{
 \mathbb C^N = \mathbb R^{2N} \ar[r] & \mathbb R\\
 \ \ \ \ z \ar@{|->}[r] &  \frac{1}{|f(z)|}
 }$$
  ``blows-up" at a point $x$ in the zero set of $f$. We attempt to measure this singularity via integration. For example, is this function square integrable in a neighborhood of $x$? 
The integral 
$$ \int \frac{1}{|f|^{2}} $$
never converges  in any small ball around $x$, but we can dampen the rate at which $\frac{1}{|f|}$ blows up by raising to a small positive  power $\lambda$. Indeed,  for sufficiently small positive  real numbers $\lambda$, depending on $f$, 
the integral
 $$ \int_{B_{\varepsilon}(x)}\frac{1}{|f|^{2\lambda}}$$ is finite,   where 
 $B_{\varepsilon}(x)$ denotes a ball of sufficiently small radius  around $x$.
 As we vary the parameter $\lambda$ from very small positive values to larger ones, there is a critical value at which 
the function  $\frac{1}{|f|^{\lambda}}$ suddenly fails to be $L^2$ locally at  $x$. This is the log canonical threshold or complex singularity exponent  of $f$ at $x$.  That is,
\begin{defi} 
The \emph{complex singularity exponent} of $f$ (at $x$) is defined as
$$ lct_{x}(f):=\sup \left\{ \lambda  \in \R_{+} \ \big|  \  {\hbox{ there exists a neighborhood $B$ of $x$ such that }}\int_{B} \frac{1}{|f|^{2\lambda}} < \infty \right\}.$$
 When the point $x$ is understood,  we denote the complex singularity exponent simply by $lct(f)$.   \end{defi} 

\vspace{0.1cm}

The following figure depicts the $\lambda$-axis: for small values of $\lambda$ the function $ z\mapsto  \frac{1}{|f(z)|}$ belongs to $L^2$ locally in a neighborhood of a point $x$; for larger $\lambda$ is does not. It is  not clear whether or not the function is integrable {\it at\/}  the complex singularity exponent; we will see shortly that it is. 

\setlength{\unitlength}{0.9mm}
\begin{center}
\begin{picture}(80,10)
\put(0,5){\vector(1,0){80}}
\put(0,5){\makebox(0,0){$[$}}
\put(33,5){\makebox(0,0){$)$}}
\put(33.7,5){\makebox(0,0){$[$}}
\put(17,8){\makebox(0,0){\small$\frac{1}{|f|^{2\lambda}}$ integrable}}
\put(57,8){\makebox(0,0){ \small $\frac{1}{|f|^{2\lambda}}$ not integrable}}
\put(90,5){\makebox(0,0){{\small $\lambda$-axis}}}
\put(34,0){\makebox(0,0){{\small $lct(f)$}}}
\end{picture}
\end{center}

\vspace{0.1cm}

This numerical invariant  is more commonly known as the {\it log canonical threshold,\/} but it is natural to use the original analytic name when approaching it from the analytic point of view. See Remark \ref{name}.

\begin{ex}\label{ex1}  If  $f$ is smooth at $x$,  then its complex singularity exponent at $x$ is 1.  Indeed, in this case the polynomial $f$  can be taken to be part of a system of local coordinates for $\mathbb C^n$ at $x$. It is then easy to compute that the integral
 $$ \int_{B(x)}\frac{1}{|f|^{2\lambda}}$$ 
 always converges on any bounded ball $B(x)$ for any positive $\lambda < 1$. Indeed, this computation is a special case of Example \ref{ex2}, below.
\end{ex}

\begin{ex}\label{ex2} 
Let $f=z_1^{a_1} \cdots z_N^{a_N}$ be a monomial in $\C[z_1, \ldots, z_N]$, which defines a singularity at the origin in $\mathbb C^N$. Let us compute its complex singularity exponent. 
By definition, we need to integrate
$$   \frac{1}{|z_1^{a_1} \cdots z_N^{a_N}|^{2\lambda}}$$
over a ball around the origin. To do so, we 
 use polar coordinates. We have each $|z_i| = r_i$ and each $dz_i \wedge d\bar z_i = r_i d r_i \wedge d\vartheta_i$.  Hence we see that 
 $$
 \int \frac{1}{|z_1|^{2a_1\lambda} \cdots |z_N|^{2a_N \lambda}}
 $$ converges in a neighborhood $B$ of the origin if and only if 
 $$
 \int_B \frac{r_1 \cdots r_N}{r_1^{2a_1 \lambda} \cdots r_N^{2a_N\lambda}} =
 \int_B \frac{1}{r_1^{2a_1 \lambda-1} \cdots r_N^{2a_N\lambda-1}}
 $$
 converges.
By Fubini's theorem, this integral converges if and only if $1-2a_i \lambda >-1$ for all $i$, that is, $\lambda < \frac{1}{a_i}$ for all $i$. Thus 
$$lct(z_1^{a_1} \cdots z_N^{a_N})=\min_i\Big\{ \frac{1}{a_i}\Big \}.$$
\end{ex}

If $f$ has ``worse" singularities, the function $\frac{1}{|f|}$ will blow up faster and the complex singularity exponent will typically be smaller. In particular, the complex singularity exponent is always less than or equal to the complex singularity exponent  of a smooth point, or one.{\footnote{One caveat: the singularity exponent  behaves somewhat differently over $\mathbb R$, due to the possibility that a polynomial's zeros are hidden over $\mathbb R$, so that $\frac{1}{|f|}$ may fail to blowup as expected or even at all!}}   Although it is  not obvious, the complex singularity exponent is always a positive rational number. We  prove this in the next subsection using Hironaka's theorem on resolution of singularities.

\subsection{\bf{Computing complex singularity exponent by monomializing.} }\label{exa3} Hironaka's beautiful theorem on resolution of singularities allows us to reduce the computation of the integral in the definition of the complex singularity exponent for any polynomial (or analytic) function to the monomial case.  Let us recall Hironaka's theorem (cf. \cite{H64}).

\begin{Thm}\label{hironaka}
Every polynomial (or analytic) function on  $\mathbb C^N$ has a monomialization. That is, 
there exists a proper birational morphism $X \overset{\pi}{\longrightarrow} \C^N$ from a smooth variety $X$  such that  both $$f \circ \pi \,\,\,\,\, {\rm { and }} \,\,\,\,\,  \Jac_\C(\pi)$$  are monomials (up to unit)  in local coordinates locally at each point of $X$. \end{Thm}

Since $X$ is a smooth complex variety, it has a natural structure of a complex manifold. Saying that $\pi$ is a morphism of algebraic varieties means simply that it is defined locally in coordinates by polynomial (hence analytic) functions, therefore $\pi$ is also a holomorphic mapping of complex manifolds. The word  ``proper" in this context can be understood in the usual analytic sense: the preimage of a compact set is compact{\footnote{Alternatively, it can be taken in the usual algebraic sense as defined in Hartshorne, \cite[Ch. 2, Section 4]{Ha}.}}.  The fact that $\pi$ is birational (meaning it has an inverse on some dense open set) is not relevant at the moment, beyond the fact that the dimension of $X$ is necessarily $N$. 

The condition that both $f \circ \pi $ and $ \Jac_\C(\pi)$ are monomials (up to unit)  locally at a point $y \in X$ means that we can find local coordinates $z_1, \ldots, z_N$ at $y$, 
such that both 
\begin{equation}\label{D}
f \circ \pi= u z_1^{a_1} \cdots z_N^{a_N},
\end{equation} 
and  the  holomorphic Jacobian{\footnote{If we write $\pi$ in local holomorphic  coordinates as $(z_1, \dots, z_N) \mapsto (f_1, \dots, f_N)$, then $\Jac_{\C}(\pi)$ is the holomorphic function obtained as  the determinant of  the $N \times N$ matrix 
$\Big(\frac{\partial f_i}{\partial z_j}\Big).$ }}
\begin{equation}\label{Jacform}
\Jac_\C(\pi)= v z_1^{k_1} \cdots z_N^{k_N}
 \end{equation} 
where $u$ and $v$ are  some regular (or analytic) functions  defined in a neighborhood  of $y$ but not  vanishing at $y$.

The properness of the map $\pi$ guarantees that the integral
$$\int\frac{1}{|f|^{2\lambda}}$$
converges in a neighborhood of the point $x$ if and only if the 
integral
$$\int\frac{\Jac_{\mathbb R}(\pi)}{|f \circ \pi |^{2\lambda}}$$
converges in a neighborhood of $\pi^{-1}(x)$, where $\Jac_{\mathbb R}(\pi)$ is the (real) Jacobian of the map $\pi$ considered as a smooth map of real $2N$-dimensional manifolds. 
Recalling that 
 $$
 \Jac_\R (\pi)=|\Jac_\C(\pi)|^2,
 $$
(see \cite[pp. 17--18]{GH}), and using that  $\pi^{-1}(x)$ is compact, Hironaka's theorem reduces the convergence of this integral to a computation with monomials in each of finitely many charts $U$  covering $X$:
$$
\int_{\pi^{-1}(B(x)) \cap U} \frac{|z_1^{k_1} \cdots z_N^{k_N}|^2}{\ \ |z_1^{a_1} \cdots z_N^{a_N}|^{2\lambda}}.
$$
Doing an analogous computation to that in Example \ref{ex2}, we can conclude that the integral is finite if and only if in each chart we have 
\begin{equation}
k_i - \lambda a_i > -1
\end{equation} 
for all $i$, or equivalently,
\begin{equation}\label{cond}
\lambda < \frac{k_i+1}{a_i}
\end{equation} 
 for all $i$. Hence
\begin{equation}\label{lct1}
 lct_x(f)=\!\!\!\!\!\!\min_{
{\scriptstyle \begin{array}{c}
\vspace{-.2cm}
\scriptstyle i\\
\scriptstyle\text{ all charts}
\end{array}}}
\!\!\!\!\!\!\left\{ \frac{k_i+1}{a_i} \right\}.
\end{equation}In particular,  remembering that the map was proper so that only finitely many charts are at issue here, we have:
\begin{Cor}
The complex singularity exponent of a complex polynomial at any point is a rational number.
\end{Cor}


\subsection{\bf Algebro-Geometric Approach}\label{febrero}   
In the world of algebraic geometry, we might attempt to measure the singularities of $f$ by trying to measure the complexity of a resolution of its singularities. Hironaka's theorem can be stated as follows:
\begin{Thm}\label{hironaka2}
Fix a polynomial (or analytic)  function $f$ on $\mathbb C^N$. There exists a proper birational morphism $X \overset{\pi}{\longrightarrow} \C^N$ from a smooth variety $X$  such that  the pull-back of $f$ defines a divisor $F_{\pi}$ whose support has simple normal crossings, and which  is also in normal crossings with the exceptional divisor (the locus of points on $X$ at which $\pi$ fails to be an isomorphism). Furthermore, the morphism $\pi$ can be assumed to be an isomorphism outside the singular set of $f$.  \end{Thm}

The proper birational morphism $\pi$ is usually called a {\it log resolution\/} of $f$ in this context. The support of the divisor defined by the    pull-back  of $f$ is simply the zero set of $f\circ \pi$. The condition that it has normal crossings means  that it is a union of smooth hypersurfaces meeting transversely. In more algebraic language, a divisor with normal crossing support is one whose equation  can be written as a monomial in local coordinates at each point of $X$. Thus
Theorem \ref{hironaka2} is really just a restatement of Theorem \ref{hironaka}.{\footnote{As stated here, Theorem \ref{hironaka2} is actually a tiny bit stronger, since the condition that we have {\it simple \/} normal crossings rules out self-crossings. The difference is immaterial to our discussion.}}

Hironaka actually proved more: such a log resolution can be constructed by a sequence of blowings-up at smooth centers. We might consider the polynomial $f$ to be ``more singular" if the number of blowings up required to resolve $f$, and their relative complicatedness, is great. However, because there is no canonical way to resolve singularities, we need a way to compare across different resolutions. This is done with the canonical divisor.

\subsection{The canonical divisor of a map.}  Fix a proper birational morphism $X \overset{\pi}\longrightarrow Y$ between smooth varieties. The  holomorphic Jacobian (determinant) $\Jac_{\C}(\pi)$ can be viewed as a regular function locally on charts of  $X$.  Its zero set (counting multiplicity) is  the  {\it canonical divisor}  of $\pi$ (or {\it relative canonical divisor} of $X$ over $Y$), denoted by $K_{\pi}$. Because the Jacobian matrix is invertible at $x \in X$ if and only if $\pi$ is an isomorphism there, the  canonical divisor of $\pi$ is supported precisely on the exceptional set $E$, which by definition consists of the points in $X$ at which $\pi$ is not an isomorphism. 
In particular,  since it is the locally the zero set of this Jacobian determinant, 
 the exceptional set $E$ is  always a codimension one subvariety of $X$. Moreover, this exceptional set is  more naturally considered as a {\it divisor:} we label each of the components of $E$ by 
 by  the order of vanishing of the Jacobian along it. This is the canonical divisor $K_{\pi}$. That is, 
$$K_{\pi} = {\rm div}(\Jac_{\C}(\pi)) =  \sum k_i E_i,$$
where the sum ranges through all of the components $E_i$ of the exceptional set $E$ and where  $k_i $ is the order of vanishing{\footnote{
These are the same $k_i$ appearing in expressions (\ref{Jacform}) as we range over all charts of $X$.  Note that there are typically many more than $N$ components $E_i$, despite the fact that in the expression (\ref{Jacform}) we were only seeing at most $N$ of them at a time in each chart.}}
of $\Jac_{\C}(\pi)$ along $E_i$. Thus we can view the canonical divisor $K_{\pi}$ as a precise ``difference" between birationally equivalent  varieties  $X$ and $Y$.

To measure the singularities of a polynomial $f$, consider a log resolution $X \overset{\pi}\longrightarrow \C^N$.  The polynomial $f$  defines a simple crossing divisor $F_{\pi}$  on $X$, namely  the zero set (with multiplicities) of the regular function $f \circ \pi$, 
 $$F_{\pi} = {\rm div}(f \circ \pi) = \sum a_i D_i 
 $$
 where the $D_i$ range through all irreducible divisors on $X$ and the $a_i$ are the orders of vanishing\footnote{Of course, the order of vanishing is zero along any irreducible divisor not in the support of $F$, so the sum is finite. Again,  these are the same $a_i$ as in formula (\ref{D}); there are typically many divisors in the support of $\pi^* F$ although in formula (\ref{D}) we see at most  $N$ in each chart.} of $f\circ \pi$ along each.
 If we denote the divisor of $f$ in $\mathbb C^N$ by $F$, then $F_{\pi}$ is simply $\pi^*F$. 
 There are two types of divisors in the support of $F_{\pi}$: the birational transforms $\widetilde F_i$ of the components of $F$, and exceptional divisors $E_i$.  All are smooth. Note that locally in charts,  both types of divisors---the  $E_i$ and the $\widetilde F_i$--- are defined by some local coordinates $z_i$ on $X$.

  Using this language, we examine our computation for the convergence  of the integral
$$\int_{B(x)}\frac{1}{|f|^{2\lambda}}. $$ The  condition  (\ref{cond}) that $k_i - \lambda a_i > -1$
is equivalent to the condition that all coefficients of the $\R$-divisor 
$$
K_\pi-\lambda F_{\pi}
$$
are greater than $-1$.
Put differently, the integral $\int \frac{1}{|f|^{2 \lambda}}$  converges in a neighborhood of $x$ if and only if the ``round up" divisor{\footnote{
 Given a divisor $D$ with real coefficients, we define the round up  $\lceil D\rceil$ as the integral divisor obtained by rounding up all coefficients of prime divisors to the nearest integer. In the same way, $\lfloor D\rfloor $ is obtained by rounding down.}}
$$
\lceil K_\pi-\lambda F_{\pi} \rceil
$$
is effective. [Strictly speaking, since we are computing the complex singularity exponent  at a particular point $x$, we should throw away any components of $K_{\pi} - \lambda F_{\pi}$
whose image on $\mathbb C^N$ does not contain $x$; that is, we should consider a log resolution of singularities only in a sufficiently small neighborhood of $x$].

Again we arrive at the following formula for the complex singularity exponent of $f$ at $x$:
\begin{Cor}
Let $\pi: X \longrightarrow \C^N$ be a log resolution of the polynomial $f$. If we write 
\begin{equation}\label{eq2}
K_{\pi} = \sum k_i D_i, \,\,\,\,\,\,\,\,\,\,\,{\rm and}\,\,\,\,\,\,\,\,\,\,\,\,\,\, {\rm div}(f \circ \pi) = \sum a_i D_i,
\end{equation}
where the $D_i$ range through all  irreducible divisors on $X$, then the complex singularity exponent or the log canonical threshold of $f$ at at point $x$ is the  minimum, taken over all indices $i$ such that $x \in \pi(D_i)$, of the rational numbers
$$ \frac{k_i+ 1}{a_i}.$$
\end{Cor} 

The complex singularity exponent is better known in algebraic geometry as the log canonical threshold.

\begin{Rmk} 
 The condition that $\lceil K_\pi-\lambda F_{\pi} \rceil
$ is effective is independent of the choice of log resolution. This follows from our characterization of the convergence of the integral but can also be shown directly using the tools of algebraic geometry (see \cite{KoM}). Although we did not motivate the study of $K_\pi-\lambda F_{\pi}$ in purely algebro-geometric terms, the $\mathbb R$- divisors $K_\pi-\lambda F $ 
 turn out to be quite natural in birational algebraic geometry, without reference to the integrals. See, for example, \cite{Ko}. In any case, our discussion shows that the definition of log canonical threshold can be restated as follows: 
\end{Rmk}

\begin{defi}\label{lct} The \emph{log canonical threshold} of  a polynomial $f$   is defined as
$$ lct_x(f):=\sup \left\{ \lambda  \in \R_{+} \ \big|\  \lceil K_{\pi} - \lambda F_{\pi}\rceil\,\,\,\, {\rm{is \,\,\,\, effective\,}} \right\},$$
where $X \overset\pi\longrightarrow \C^N$ is any log resolution of $f$ (in a neighborhood of $x$), $K_{\pi}$ is its relative canonical divisor, and $F_{\pi} $ is the divisor on $X$ defined by $f \circ \pi$.  We can also define the global log canonical threshold by taking $\pi$ to be a resolution at all points, not just in a neighborhood of $x$.
\end{defi}

\begin{Rmk}
Note that loosely speaking, the more complicated the resolution, the more likely $\lambda$ will have to be small in order make $K_{\pi} - \lambda F_{\pi}$ close to effective.  This essentially measures the complexity of the pullback of $f$ to the log resolution. The presence of the  $K_{\pi}$ term accounts for the added multiplicity that would have been present in {\it any\/} resolution, because of the nature of blowing up $\C^N$ to get $X$, thus ``standardizing" across different resolutions.

It is also clear from this point of view that $\lceil K_{\pi} - \lambda F_{\pi}\rceil$ is always effective for very small (positive) $\lambda$, and that as we enlarge $\lambda$ it stays effective until we suddenly hit the log canonical threshold of $f$, at which point at least one coefficient is exactly negative one.
\end{Rmk}

\begin{Rmk}\label{name}  The name {\it log canonical}  comes from birational geometry. A pair $(Y, D)$ consisting of a $\mathbb Q$-divisor on a smooth variety $Y$ is said to be log canonical if, for any  proper birational morphism $ \pi: X \longrightarrow Y$
with $X$ smooth (or equivalently, any fixed log resolution), the divisor $ K_{\pi} - \pi^*D $ has all coefficients $\geq -1.$  This condition is independent of the choice of $\pi$ (see \cite{KoM}). Thus the log canonical threshold of $f$ at $x$  is the supremum, over positive $\lambda \in \R$ such that $(\C^n, \lambda {\rm div}(f))$ is log canonical in a neighborhood of $x$.  
\end{Rmk}

\begin{ex} The log canonical threshold of any complex polynomial  $f$ is bounded above by one. Indeed, suppose for simplicity that $f$ is irreducible, defining a hypersurface $D$ with isolated singularity at $x$.
 Let $\pi: X \longrightarrow \C^N$ be a log resolution of $f$. We have
$$
K_{\pi} = \sum k_i E_i,
$$ where all the $E_i$ are exceptional, and 
$$
{\rm div} (f \circ \pi) = \sum a_i E_i + \widetilde D,
$$ where the $E_i$ are exceptional and  $\widetilde D$ is the birational transform of $D$ on $X$.
Then the log canonical threshold is the minimum value of 
\begin{equation}\label{lct2}
\!\!\!\!\!\!\min_{
{\scriptstyle \begin{array}{c}
\vspace{-.2cm}
\scriptstyle i\\
\scriptstyle\text{ }
\end{array}}}
\left\{ \frac{k_i+1}{a_i}, 1 \right\}
\end{equation}
as we range through the exceptional divisors of $\pi$.
More generally, the argument adapts immediately to show that if  $f$ factors into irreducibles as $f = f_1^{b_1}f_2^{b_2} \cdots f_t^{b_t}$, then the log canonical threshold is bounded above by the minimal value of $\frac{1}{b_i}$.
\end{ex}

\subsection{\bf Computations of Log Canonical Thresholds.}  

The canonical divisor of a morphism plays a starring role in birational geometry, and in particular, as we have seen, in the computation of the log canonical threshold. 
Before computing some more examples, we isolate two helpful properties of $K_{\pi}$.
\begin{fact}\label{f1}
Let $X \overset{\pi}{\longrightarrow}Y $ be the blow-up along a smooth subvariety of codimension $c$ in the smooth variety $Y$. Then the relative canonical divisor is
$$K_\pi=(c-1)E,$$
where $E$ denotes the exceptional divisor of the blow-up.
\end{fact}

\begin{fact} Consider a sequence of  proper birational morphisms  $X_3\overset{\pi}{\longrightarrow} X_2\overset{\nu}{\longrightarrow} X_1$, where all the $X_i$ are smooth. Then,
$$K_{\nu\circ \pi}=\pi^*K_\nu + K_\pi.$$
\end{fact}
The proof of both these facts are easy exercises in local coordinates, and left to the reader.

\begin{ex}\label{cusp0} {\it A cuspidal singularity.} We compute the log canonical threshold of the cuspidal curve $D$ given by  $f=x^2-y^3$  in $\C^2$ (at the origin, its unique singular point). The curve is easily resolved (that is, the polynomial $f$ is easily monomialized)  by a sequence of three point blowups at points:  $X_3 \overset{\psi}{\longrightarrow}  X_2 \overset{\nu}{\longrightarrow}  X_1\overset{\phi}\longrightarrow  \C^N  $,  whose composition we denote by $\pi$, and which create exceptional divisors $E_1, E_2$ and $E_3$ respectively.{\footnote{In a slight, but very helpful, abuse of terminology, we use the same symbol to denote an irreducible divisor and its birational transform on any model.}}
[Here $\phi$ is the blowup at the origin, $\nu$ is the blowup of the unique intersection point with the birational transform of $D$ with $E_1$, and $\psi$ is the blowup of the unique intersection point of the birational transform of $D$ on $X_2$ with $E_2$]. There are four relevant divisors on $X_3$ to consider, the three exceptional divisors   $E_1, E_2$ and $E_3$, and  the birational transform of $D$ on $X_3$.  Using the two facts above, it is easy to compute that 
$$
K_{\pi} = E_1 + 2E_2 + 4E_3
$$
and
$$
F=\Div (f \circ \pi)= D+  2E_1+3E_2 + 6E_3
$$
Hence $lct(f)=\frac{5}{6}$. 
\end{ex}

\vspace{0.1cm}

\begin{ex} As an  exercise, the reader can compute that for  $f=x^m - y^n$ with $\gcd(m,n)=1$, then $lct(f)=\frac{1}{m} + \frac{1}{n}$. The resolution is constructed as in Example \ref{cusp0} but may require a few more blowups to resolve. 
\end{ex}

\vspace{0.2cm}

\begin{ex} Let $f$ be a homogenous polynomial  of degree $d$ in $N$ variables, with an isolated singularity at the origin. Then  $lct(f)=\frac{N}{d},$  if $d \geq N$ and $1$ otherwise.
  Indeed,
one readily checks that blowing up the origin, we obtain a log resolution, $X \overset{\pi}\longrightarrow \C^N$, with one exceptional component $E$. 
Using Fact \ref{f1}  above, we compute that 
$$
K_{\pi} = (N-1) E.
$$ Also, the divisor $D$ defined by $f$ pulls back to 
$$
F  = dE +  D,
$$ where 
again, $D$ denotes also the birational transform of $D$ on $X$. Thus
$$lct(f) =\min\left\{\frac{(N-1)+1}{d}, \frac{ 0+ 1}{1}\right\} =  \min\left\{\frac{N}{d}, {1}\right\}.$$
\end{ex}

\begin{Rmk}
The log canonical threshold  describes the singularity but it does not characterize it. For example, the previous example gives examples of numerous non-isomorphic non-smooth points whose log canonical threshold is one, the same as a smooth point.
\end{Rmk}

\begin{Rmk}
In general it is hard to compute the log canonical threshold, but there are algorithms to compute it in special cases such as the monomial case (\cite{How}),  the toric case (\cite{Bl}) or the case of two variables (\cite{Tu}). In all these cases, the reason the log canonical threshold can be computed is that a resolution of singularities can be explicitly understood.
\end{Rmk}

\subsection{\bf Multiplier ideals and Jumping Numbers.}
Our definition of log canonical threshold leads naturally to a family of richer invariants called the {\it multiplier ideals} of $f$,  which are  ideals in the polynomial ring indexed by the positive real numbers.

Again, multiplier ideals can be defined analytically or algebro-geometrically.

\begin{defi}(Analytic Definition, cf. \cite{De1}). Fix $f \in \C[x_1, \ldots, x_n]$. For each $\lambda \in \R_+$, define the \emph{multiplier ideal} of $f$ as
$$ \mc{J}( f^\lambda)=\left\{ h \in \C[x_1, \ldots,x_n]\ \big|\  {\hbox{ there exists a neighborhood $B$ of $x$ such that}}   \int_{B}\frac{|h|^2}{\ |f|^{2\lambda}} < \infty\right\}.$$
Thus the multiplier ideals consist of functions that can be used as ``multipliers" to make the integral converge. It is easy to check that this set  $\mc{J}(f^\lambda)$ is in fact an ideal of the ring $\C[x_1, \dots, x_n].$
 \end{defi}

Equivalently,  we define the multiplier ideal  in an algebro-geometric context using a log-resolution.

\begin{defi} (Algebro-Geometric Definition). \label{numero0}
Fix $f \in \C[x_1, \ldots, x_n]$. 
For each $\lambda \in \R_+$, define the \emph{multiplier ideal} of $f$ as 
$$
\mc{J}( f^\lambda)=\pi_*\calo_X(\lceil K_\pi-\lambda F_{\pi}\rceil),
$$ 
where $X \overset{\pi}\longrightarrow \C^n$ is a log-resolution of $f$, $K_{\pi}$ is its relative canonical divisor and  $F_{\pi}$ is the divisor on $X$ determined by $f\circ \pi$.
These are the polynomials whose pull-backs to $X$ have vanishing no worse than that of the divisor $K_{\pi} - \lambda F_{\pi}$.
 \end{defi}

Recall that if $D$ is a divisor on a smooth variety $X$, the notation $\calo_X(D)$ denotes the sheaf of rational functions $g$ on $X$ such that div$(g) + D $ is effective. Thus in concrete terms, the multiplier ideal  is
$$
\pi_*\calo_X(\lceil K_\pi-\lambda F\rceil)
= 
\{ h \in \C[x_1, \ldots, x_n] \,\, | \,\, {\rm{div}}( h \circ \pi) + \lceil K_\pi-\lambda F\rceil \geq 0\}.
$$ 
It is straightforward to check that the argument we gave for 
  translating the analytic  definition of the log canonical threshold into algebraic geometry can be used to see that these two definitions of multiplier ideals are equivalent. Moreover, this also shows that Definition \ref{numero0} is independent of the choice of resolution. For a direct algebro-geometric proof, see \cite[Chapter 9]{Laz}.
   
   \begin{Rmk} Because $K_{\pi}$ is an integral divisor, we have $\lceil K_\pi-\lambda F_{\pi}\rceil = K_{\pi} - \lfloor \lambda F_{\pi}\rfloor$. We caution the reader delving deeper into the subject, however, that 
   when the notion of multiplier ideals and log canonical thresholds are generalized to divisors on singular ambient spaces, there are situations in which $K_{\pi}$ is a non-integral $\mathbb Q$-divisor. In this case, we can not assume that  $\lceil K_\pi-\lambda F_{\pi}\rceil = K_{\pi} - \lfloor \lambda F_{\pi}\rfloor$.
   \end{Rmk}
   
   \begin{Prop}\label{basic} Fix a polynomial $f$ and view its multiplier ideals $ \mc{J}(f^\lambda)$ as a family of ideals varying with $\lambda$. Then the following properties hold:
   \begin{enumerate}
    \item For $\lambda \in \R_+$ sufficiently small, $ \mc{J}(f^\lambda)$  is the unit ideal. 
    \item If $\lambda > {\lambda}'$, then $ \mc{J}(f^{\lambda'}) \supset  \mc{J}(f^\lambda).$
  \item  The log canonical threshold of $f$ is 
 $$lct(f)=\sup \{ \lambda \ | \ \J(f^\lambda)=(1) \}.$$
  \item For each fixed value of $\lambda$,   we have  $\J(f^\lambda)= \J(f^{\lambda+\varepsilon})$ for small enough positive $\varepsilon$. (How small is small enough  depends on $\lambda$).
 \item  There exist certain $\lambda\in \R_+$ such that $\J(f^{\lambda - \varepsilon})\supsetneq \J(f^{\lambda})$ for all positive values of $\varepsilon$.
\end{enumerate}
\end{Prop}

All of these properties are easy to verify, thinking of what happens with the rounding up of the  divisor $K_{\pi} - \lambda F$ as $\lambda$ changes.
As we imagine starting with a very small $\lambda$ and increasing it,  the properties above can be summarized by the following diagram:

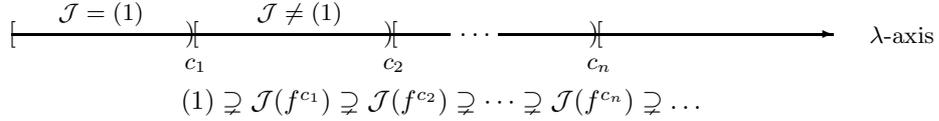
\setlength{\unitlength}{0.9mm}
\begin{center}
\begin{picture}(120,18)
\put(0,10){\makebox(0,0){$[$}}
\put(0,10){\line(1,0){64}}
\put(68,10){\makebox(0,0){$\dots$}}
\put(71,10){\vector(1,0){49}}
\put(130,10){\makebox(0,0){{\small $\lambda$-axis}}}

\put(26,10){\makebox(0,0){$)$}}
\put(26.8,10){\makebox(0,0){$[$}}
\put(26.8,5.5){\makebox(0,0){\small $c_1$}}
\put(13,13){\makebox(0,0){\small${\J=(1)}$}}

\put(55,10){\makebox(0,0){$)$}}
\put(55.8,10){\makebox(0,0){$[$}}
\put(55.8,5.5){\makebox(0,0){\small $c_2$}}
\put(42,13){\makebox(0,0){\small $\J\not=(1)$}}

\put(85,10){\makebox(0,0){$)$}}
\put(85.8,10){\makebox(0,0){$[$}}
\put(85.8,5.5){\makebox(0,0){\small $c_n$}}

\put(63,0){\makebox(0,0){$(1)\supsetneq\J(f^{c_1})\supsetneq\J(f^{c_2})\supsetneq\dots\supsetneq\J(f^{c_n})\supsetneq \dots$}}
\end{picture}

\vspace{0.1cm}
\small{There are certain critical exponents $c_i$ for which the multiplier ideal ``jumps." }
\end{center}

The critical numbers  $\lambda$ described in (5) and denoted by $c_i$ in the diagram give a sequence of numerical invariants refining the log canonical threshold, which is the smallest of these. Formally: 

\begin{defi} The jumping numbers of $f \in \C[x_1, \dots, x_n]$ are the positive real numbers $c$ such that $\J(f^{c}) \subsetneq \J(f^{c-\epsilon})$ for all positive $\epsilon$.
\end{defi}

\begin{Prop}
The jumping numbers of $f \in \C[x_1, \dots, x_n]$ are discrete and rational.
\end{Prop}
\begin{proof}
Let $\pi: X \rightarrow \C^N$ be a log resolution of $f$, with $K_{\pi} = \sum k_i D_i$ and 
$F_{\pi}= \sum a_i D_i$, as before.  It is easy to see that the critical values of $\lceil K_{\pi} - \lambda F_{\pi} \rceil$ occur only
when $k_i - \lambda a_i \in \N$. That is, the 
jumping numbers are a subset of the numbers $\{\frac{k_i + m}{a_i}\}_{m \in \N}.$  In particular,  they are discrete and rational.
\end{proof}

 Although an infinite sequence, the jumping numbers are actually determined by finitely many: 
\begin{Thm}\label{bs}{\rm (See }{\it{ e. g.}}  \cite[Theorem 9.6.21]{Laz}{\rm ).}
Fix a polynomial $f$ as above. Then
$$\J(f^{1+\lambda})=(f)\J(f^\lambda),$$
for $\lambda \geq 0$. In particular, a positive real number $c$ is a jumping number if and only if $c+1$ is a jumping number.
\end{Thm}

Thus the jumping numbers are periodic and  completely determined by the finite  set of jumping numbers less than or equal to 1.

\begin{Rmk}
It is quite subtle to determine which ``candidate jumping numbers" of the form $\frac{k_i+1}{a_i}$ are actual jumping numbers. When $f$ is a polynomial in 2 variables, there is some very pretty geometry behind understanding this question. See \cite{SmT}, \cite{Tu}.
\end{Rmk}

\begin{Exercise} Show that $1$ is a jumping number of every polynomial. [Hint: if $f$ is irreducible, the jumping number 1 is contributed  by the birational transform of div$(f)$ on the log resolution].
\end{Exercise}

\begin{Exercise} Show that the jumping numbers of a smooth $f$ are the natural  numbers $1, 2, 3, \dots$
\end{Exercise}

\begin{Exercise} Using the resolution described in Example \ref{cusp0}, show that the multiplier ideals of $f = x^2 - y^3$ are as follows:
\begin{enumerate}
\item $\J(f^{\lambda}) $ is trivial for values of $\lambda$ less than $\frac{5}{6}$.
\item  $\J(f^{\lambda}) = \mf{m} = (x, y)$ for   $\frac{5}{6} \leq \lambda < 1$;
\item  $\J(f^{\lambda} ) = (f)$ for   $1\leq \lambda < \frac{11}{6}.$
\end{enumerate} Using Theorem \ref{bs}, describe the multiplier ideal of $x^2 - y^3$ for any value of $\lambda$.
\end{Exercise}

\begin{Exercise} (Harder) Show that the multiplier ideals of $f = x^2 - y^5$ are as follows:
\begin{enumerate}
\item $\J(f^{\lambda}) = R$ is for values of $\lambda$ less than $\frac{7}{10}$.
\item  $\J(f^{\lambda}) = \mf{m} = (x, y)$ for   $\frac{7}{10} \leq \lambda < \frac{9}{10}$;
\item  $\J(f^{\lambda}) =  (x, y^2)$ for   $\frac{9}{10} \leq \lambda < 1$;
\item  $\J(f^{\lambda} ) = (f)$ for   $1\leq \lambda < \frac{17}{10}.$
\end{enumerate}
\end{Exercise}

\begin{Rmk}
The jumping numbers turn out to be related to many other well-studied invariants. For example, it is shown in \cite{ELSV} that the jumping numbers of $f$ in the interval $(0,1]$  are always (negatives of) roots of the so-called Bernstein-Sato polynomial $b_f$ of $f$.  The jumping numbers can also be viewed in terms of the Hodge Spectrum arising from the monodromy action on the cohomology of the Milnor fiber of $f$  \cite{Bu1}. 
\end{Rmk}

Multiplier ideals have many additional properties, in addition to many deep applications  which we don't even begin to describe. Lazarsfeld's book \cite{Laz} gives an idea of some of these.
\vspace{0.1cm}

\section{Positive characteristic: The Frobenius map and $F$-thresholds}

A natural question arises: What about positive  characteristic? 

Fix  a polynomial $f$ over a perfect field $k$. We wish to measure the singularity of $f$ at some point where it vanishes.
For concreteness and with no essential loss of generality, say the field is $\F_p$ and the point is the origin, so that 
  $ f \in \mf{m} = (x_1, \dots, x_n)  \subset \F_p[x_1, \ldots, x_n]$.  
  How can we try to define an analog of log canonical threshold? 
  In characteristic zero, we used real analysis to  control the growth of the function $\frac{1}{|f|^{\lambda}}$ as we approached the singular points of $f$.  But in characteristic $p$, can we even talk about taking fractional powers of $f$?   
Remarkably, 
the Frobenius map gives us a tool for raising polynomials to non-integer powers, and for considering their behavior near $\mf{m}$.

\subsection{\bf The Frobenius map.}

Let $R$ be any ring of characteristic $p$, with no non-zero nilpotent elements.

\begin{defi}  The \emph{Frobenius map} $F$ is the ring homomorphism 
$$\xymatrix@C=3pc@R=0pc{
R\ar[r]^F & R\\
r \ar@{|->}[r] & r^p. 
}
$$
The image is the subring $R^p$ of $p$-th powers of $R$, which is of course  isomorphic to $R$ via $F$ (provided $R$ has no non-trivial nilpotents, so that $F$ is injective.)
\end{defi}
Nothing like this is true in characteristic zero. The point is that in characteristic $p$, the Frobenius map respects addition [$(r+s)^p=r^p+s^p$ for all $r, s \in R$], because the binomial coefficients
$\binom{p}{j}$ are congruent to $0$ modulo $ p$ for every $1\leq j\leq p-1$.

 By iterating the Frobenius map we get an infinite chain of subrings of $R$:
   \begin{equation}\label{tele}
 R \supset R^p \supset R^{p^2} \supset R^{p^3} \supset \cdots,
\end{equation}
 each isomorphic to $R$.  Alternatively,  we can imagine adjoining $p$-th roots: inside a fixed algebraic closure of the fraction field of $R$, for example, each element of $R$ has a {\it unique\/} $p$-th root. Now the ring inclusion $R^p \subset R$ is equivalent to the ring inclusion $R \subset R^{\frac{1}{p}}$; the Frobenius map gives an isomorphism between these two chains of rings. Iterating we have an increasing but essentially equivalent chain of rings
 $$
R \subset R^{\frac{1}{p}} \subset R^{\frac{1}{p^2}} \subset \dots.
$$ 
Viewing these $R^{\frac{1}{p^e}}$ as $R$-modules, it turns out that   a remarkable wealth of information about singularities is revealed by their $R$-module structure as $e \rightarrow \infty. $ Let us consider an example.

\begin{ex}  
Let $R=\F_p[x]$. The subring of $p$-th powers  is the ring  $\F_p[x^p]$ of polynomials 
in $x^p$, and similarly the overring of $p$-th roots is   $\F_p[x^{\frac{1}{p}}].$
Given any polynomial $g(x)\in\F_p[x]$, there is a unique way to write
$$g(x)=g_0(x^p)\cdot 1+g_1(x^p)\cdot x+\cdots+g_{p-1}(x^p)\cdot x^{p-1},$$
where each $g_j(x^p)\in R^p$. In fancier language, 
$\F_p[x]$ is a free  $\F_p[x^p]$ module on the basis $\{1, x, x^2, \dots, x^{p-1}\}$.
That is, there is an $\F_p[x^p]$-module homomorphism
$$
\F_p[x]\cong R^p\oplus R^p\cdot x\oplus\cdots\oplus R^p\cdot x^{p-1}.
$$
Iterating, we see that $\F_p[x]$ is free over $\F_p[x^{p^e}]$ on the basis  $\{1, x, x^2, \dots, x^{p^e-1}\}$. Equivalently, each  $\F_p[x^{\frac{1}{p^e}}]$ is a free $R$-module on the basis 
 $\big\{1, x^{\frac{1}{p^e}}, x^{\frac{2}{p^e}}, \dots, x^{\frac{p^e-1}{p^e}}\big\}$.
\end{ex}

\begin{ex}  Similarly, if $R$ is the polynomial ring  $\F_p[x_1,\ldots,x_n]$, 
then $R$ is  a free module over $R^{p^e} = \F_p[x_1^{p^e},\dots,x_n^{p^e}]$ with basis 
$$\left\{x_1^{a_1}\cdots x_n^{a_n}\right\}_{0\leq a_j\leq p^e-1,\ 1\leq j\leq n}.$$ \end{ex}

\vspace{0.2cm}
The freeness of the polynomial ring over its subring of $p$-th powers is no accident, but rather reflects the fact that the corresponding affine variety is smooth. 
The Frobenius map can be used to detect singularities quite generally:

\begin{Thm}{\rm (}Kunz, \cite[Theorem 2.1]{Ku}{\rm ).} Let $R$ be a ring of prime characteristic $p$ 
without nilpotent elements. Then  $R$ is regular if and if only the Frobenius map is flat.
\end{Thm}

Let us put Kunz's theorem more concretely in the case  we care about---the case where the Frobenius map $R\overset{F}{\lra} R$ is finite,  that is, when $R $ is  {\it finitely generated\/}  as an $R^p$-module.{\footnote{
All rings in which we are interested in this survey satisfy this condition. It is easy to check, for example, that if $R$ is finitely generated over a perfect field, or a localization of such, then the Frobenius map is finite. Similarly, so do rings of power series  over perfect fields. }} In this case, Kunz's theorem says that  that $R$ is regular if and only if  $R$ is a locally free $R^p$-module, or equivalently, if and only if $R^{1/p}$ is locally free as an $R$-module.

This leads to the natural question:  if $R$ is {\it not\/} regular, can we use Frobenius to measure its singularities? The answer is a resounding YES. This is the topic of a large and active body of research in ``$F$-singularities" which classifies singularities according to the structure of the chain of $R$-modules
$$
R \subset R^{\frac{1}{p}} \subset R^{\frac{1}{p^2} }\subset \dots.
$$
The $F$-threshold, which we now discuss, is only the beginning of a long and beautiful  story.

\subsection{$F$-threshold} Now we fix a polynomial $f$ in the ring $R=\F_p[x_1,\ldots,x_n]$.  For a rational number of the form $c=\frac{a}{p^e}$, we can consider the fractional power $f^c = f^{\frac{a}{p^e}}$ as an element of the overring $R^{1/p^e} =  \F_p[x_1^{1/p^e},\ldots,x_n^{1/p^e}] $.  This allows us to ``take fractional powers" of polynomials, analogously to what the analysis allowed us to do in Section 1, at least if we restrict ourselves to fractional powers whose denominators are powers of $p$.

In the analytic setting, we tried to measure how badly the function $\frac{1}{|f|^{\lambda}}$  ``blows up" at the singular point using integrability---this led to the complex singularity index or log canonical threshold. In this characteristic $p$ world, we can not integrate, nor does even absolute value make sense. Amazingly, however, the most naive possible way to 
talk about the function $\frac{1}{f^{c}}$ ``blowing up" {\it does \/} lead to a sensible invariant, which turns out to be very closely related to the complex singularity index. Indeed, 
we can agree that  $\frac{1}{f^c}$ certainly {\it does not\/} blow up at any point where the denominator does not vanish. 

Recall that each $R$ module $M$ can be interpreted as a coherent sheaf on the affine scheme $\Spec R$, in which case   each element  $s \in M$ is interpreted as a section of this coherent sheaf.
 Grothendieck defined the ``value" of a section $s$ at the point $\mf{P} \in \Spec R$ to be the image of $s$ under the natural map from $M$ to $M \otimes L$, where $L$ is the residue field at $\mf{P}$ \cite[Chap 2.5]{Ha}. In particular, the ``function"  $f^c$ (when $c$ is a rational number of the form $\frac{a}{p^e}$) is an element of the $R$-module $R^{1/p^e}$ (for some $e$),  and as such its ``value" at the  point $\mf{m}$ is zero if and only if $f^c \in \mf{m}R^{1/p^e}$.

So, given that integration does not make sense, we can at least  look at values of $c$ for which $\frac{1}{f^c}$  ``does not blow up at all,"  and take the supremum over all such $c$.  
This extremely naive attempt to mimic the analytic definition then leads to the following definition.
\begin{defi}\label{FTdef}
The \emph{$F$-threshold} of $f \in  \F_p[x_1,\ldots,x_n]$ at the maximal ideal $\mf{m} = (x_1, \dots, x_n)$ is defined as
$$
FT_{\mf{m}}(f)  =  \sup\Big\{c=\frac{a}{p^e} \in \mathbb Z\Big[\frac{1}{p}\Big] \,\, \ \big|\ f^c \not\in \mf{m}R^{1/p^e}\Big\}.
$$
\end{defi}
Amazingly, this appear to be the ``right" thing to do! Although we have stated the definition for polynomials over $\F_p$, any perfect field $k$, or indeed, any field $k$ of characteristic $p$ such that $[k:k^p]$ is finite works just as well.

Let us check that this definition is independent of how we write $c$. First note that 
 viewing  $f^c$ as an element of the free $R$-module  $R^{1/p^e}$,
 we can write it   uniquely as an $R$-linear combination of the basis elements
$$\{x_1^{\frac{a_1}{p^e}}\cdots x_n^{\frac{a_n}{p^e}}\}_{0\leq a_j\leq p^e-1},$$
which we abbreviate $\{\mathbf{x}^{A/p^e}\}$.
\begin{equation}\label{FTex}
f^c = f^{a/p^e}=\sum r_A \mathbf{x}^{A/p^e},
\end{equation}
for some uniquely determined 
$r_A\in R$. 
 So, an equivalent formulation of the  
$F$-threshold of $f \in  \F_p[x_1,\ldots,x_n]$ at the maximal ideal $\mf{m} = (x_1, \dots, x_n)$ is 
$$
FT_{\mf{m}}(f) = \sup\Big\{c=\frac{a}{p^e} \in \mathbb Z\Big[\frac{1}{p}\Big] \,\, \ \big|\ f^c \hbox{ has some coefficient } r_A\not\in \mf{m}\Big\},
$$
where the coefficients $r_A$ are as  in Equation (\ref{FTex}) above.

It is now easy to see that this supremum is independent of the way we write $c$. That is, 
 if we instead had written $c= \frac{ap}{p^{e+1}} $ and viewed $f^c$ as an element in the larger ring $R^{\frac{1}{p^{e+1}}}$, then when expressed $f^c$ uniquely as an 
$R$-linear combination of the basis elements $\mathbf{x}^{{A'/p^{e+1}}}$ for $R^{\frac{1}{p^{e+1}}}$, 
 the coefficients $r_{A'}$ that appear are the same elements of $R$ as appearing in expression (\ref{FTex}).

 By Nakayama's Lemma,  $f^c\not\in \mf{m} R^{1/p^e}$ if and only if 
 $f^c$  is  part of a minimal generating set for the $R$-module $R^{1/p^e}$ (after localizing at $\mf{m}$). So equivalently:
\begin{defi}
The \emph{$F$-threshold} of $f \in  \F_p[x_1,\ldots,x_n]$  at $\mf{m}$ is
\begin{align*}
FT_{\mf{m}}(f)&
=\sup\Big\{c=\frac{a}{p^e}   \in \mathbb Z\Big[\frac{1}{p}\Big] \,\, \big|\ f^c \hbox{ is part of a free basis 
for }\,\, R_{\mf{m}}^{\frac{1}{p^e}} {\rm over} \,\,\, R_{\mf{m}}\Big\}.\\
\end{align*}
\end{defi}

\begin{ex}\label{1}
The $F$-threshold of any polynomial is always bounded above by one. Indeed, let $f$ be any polynomial in $\mf{m}$. 
Since $f^1 = f \cdot 1 \in \mf{m}R^{\frac{1}{p^e}}$ for all $e$, we see that $f^{1}$ is never part of a basis for $R^{1/p^e}$ over $R$. We must raise $f$ to numbers {\it less than one\/} to get a basis element. Thus $FT_{\mf{m}}(f) \leq 1$ always.
\end{ex}

\begin{ex}
 Assume that $f$ is non-singular at $\mf{m}$. Then $f$ is part of a local system of regular parameters at $\mf{m}$, that is, one of the minimal generators for the ideal $\mf{m}$. Changing coordinates, we can assume $f = x_1$. 
For any $p^e$,  note that $$f^{\frac{p^e-1}{p^e}} = x_1^{\frac{p^e-1}{p^e}}$$
is part of a free basis for $R_\mf{m}^{\frac{1}{p^e}}$ over $R_\mf{m}$. This shows that 
 the $F$-threshold is greater or equal than $ \frac{p^e-1}{p^e}$ for all $e\geq 1$. In other words, taking the limit we see that the $F$-threshold at a smooth point  is bounded below by one. Combining with the previous example, we conclude that the $F$-threshold of a smooth point is exactly one.

\end{ex}

\begin{ex} Take $f=xy\in\F_p[x,y]$, then $(xy)^{\frac{p^e-1}{p^e}}$ is also part of a basis of the free $R$-module $R^{1/p^e}$. The same argument applies  here to show that $FT_\mf{m}(xy)=1$. In general, $FT_\mf{m}(x_1\cdots x_\ell)=1$, which is to say, the $F$-threshold of a simple normal crossings divisor is always one.  In particular, this shows that $FT_\mf{m}(f)=1$ does not imply  that $f$ is non-singular.
\end{ex}

\begin{ex}\label{44}
 Consider $f=x^m\in \F_p[x]$. Then, $(x^{m})^{\frac{a}{p^e}}$ is part of a free basis if and only if $\frac{m a}{p^e}\leq \frac{p^e-1}{p^e}$.
So $f^{\frac{a}{p^e}}$ is part of a free basis if and only if $\frac{a}{p^e} \leq \frac{1}{m} - \frac{1}{ m p^e}$.
Taking the limit, this leads to $FT_\mf{m}(x^m)=\frac{1}{m}$. In general, 
$$FT_\mf{m}(x_1^{a_1}\cdots x_\ell^{a_\ell})=\min\Big\{\frac{1}{a_i}\Big\}.$$
\end{ex}

Examples (\ref{1}) through (\ref{44})  indicate that  the $F$-threshold has many of the same features as the  log canonical threshold. Is the $F$-threshold capturing exactly ``the same" measurement of singularities as the log canonical threshold?  If a polynomial has integer coefficients, do we get the same value of the $F$-threshold modulo $p$ for all $p$? Of course, the ``same" polynomial can be more singular in some characteristics, so we expect not. But does the $F$-threshold for ``large  $p$" perhaps agree with the log canonical threshold? The following example is typical:

\begin{ex}\label{cuspide}
Consider the polynomial $f=x^2+y^3$, which we can view as a polynomial over any of the fields $\F_p$ (or $\C)$. Its $F$-threshold depends on the characteristic:
$$
FT_\mf{m}(f)=\left\{
\begin{array}{ll}
1/2 & \hbox{ if } p=2,\\
2/3 & \hbox{ if } p=3,\\
5/6 & \hbox{ if } p\equiv 1\mod 6,\\
5/6-\frac{1}{6p} & \hbox{ if } p\equiv 5\mod 6.
\end{array}\right.
$$
Viewing $f$ as a polynomial  over $\C$, we computed in Example \ref{cusp0} that its log canonical threshold is $\frac{5}{6}$. Interestingly, we see that as $p \longrightarrow \infty$, the $F$-thresholds approach the log canonical threshold. Also, there are some (in fact, infinitely many) characteristics where the $F$-threshold agrees with the log canonical threshold.  On the other hand, there are other characteristics where the polynomial is ``more singular" than expected, as reflected by a smaller $F$-threshold. For example, in characteristics $2$ and $3$, of course, we  expect ``worse" singularities, and indeed, we see the $F$-threshold is smaller in these cases.  But also the $F$-threshold detects a worse singularity for this curve in characteristics congruent to $ 5$ mod $6$, reflecting subtle number theoretic issues in that case (see \cite[Question 3.9]{MTW}).

\begin{ex} Daniel Hern\'andez has computed many examples of $F$-thresholds in his PhD thesis (c.f.  \cite{Her}), including any  ``diagonal" hypersurfaces $x_1^{a_1} + \cdots + x_n^{a_n}$ (see \cite{Her1}). There is also  an algorithm  to compute the $F$-threshold of any binomial as well; see  \cite{Her2}, \cite{ShT}. See also \cite{MTW}.
\end{ex}

\begin{ex}
Let $f\in\Z[x,y,z]$  be homogeneous of degree 3 with an isolated singularity. In particular, $f$ defines an elliptic curve in $\Proj^2$ over $\Z$.  The $F$-threshold has been computed in this case by  B. Bhatt  \cite{Bh}: 
$$FT(f)=
\left\{
\begin{array}{ll}
\displaystyle 1 &  \hbox{ if } E\in\Proj_{\F_p}^2 \hbox{ is \emph{ordinary}},\\
\displaystyle 1-\frac{1}{p} & \hbox{ if } E\in\Proj_{\F_p}^2 \hbox{ is \emph{supersingular}}.
\end{array}\right.$$
Again we see that ``more singular" polynomials have smaller $F$-thresholds. As in Example \ref{cuspide}, there are infinitely many $p$ for which the log canonical threshold and the F-threshold agree, and infinitely many $p$ for which they do not, by a result of Elkies \cite{Elk}.
\end{ex}

\subsection{Comparison of F-threshold and multiplicity.}
To compare the F-threshold with the multiplicity, we rephrase the definition still one more time. Let us first recall a well-known notation: for an ideal $I$ in an ring $R$ of characteristic $p$, let  $I^{[p^e]}$ denote the ideal of $R$ generated by the $p^{e}-th$ powers of the elements of $I$. That is, $I^{[p^e]}$ is the expansion of $I$ under Frobenius $R \rightarrow R$ sending $r \mapsto r^{p^e}$.  

\begin{defi}
The F-threshold of $f \in k[x_1, \dots, x_n]_{\mf{m}}  = R$ is
\begin{align*}
FT_{\mf{m}}(f) & =  \sup\Big\{c=\frac{a}{p^e} \in \mathbb Z\Big[\frac{1}{p}\Big] \,\, \ \big|\ f^a \not\in \mf{m}^{[p^e]}\Big\},\\
& = \inf \Big\{c=\frac{a}{p^e} \in \mathbb Z\Big[\frac{1}{p}\Big] \,\, \ \big|\ f^a \in \mf{m}^{[p^e]}\Big\}.\\
\end{align*}
\end{defi}
This is patently the same as Definition \ref{FTdef},  simply by raising to the $p^e$-th power.

On the other hand, the multiplicity of $f$ at $\mf{m}$ is defined as the largest $n$ such that $f \in \mf{m}^n$. It is trivial to check that this is equivalent to
$$
{{\rm mult}_{\mf{m}}(f)} =  sup \Big\{\frac{t}{a} \in \mathbb Q \,\, \ \big|\ f^a \in \mf{m}^t\Big\}.
$$
That is, the formula that computes the F-threshold is similar to  formula that computes the {\it reciprocal\/} of the  multiplicity, but with ``Frobenius powers" replacing  ordinary powers:
$$
\frac{1}{{\rm mult}_{\mf{m}}(f)} =  inf \Big\{\frac{a}{t} \in \mathbb Q \,\, \ \big|\ f^a \in \mf{m}^t\Big\}.
$$
 It is also not hard to check in all these  cases that the infimum (supremum) is in fact a limit.

This similarity allows us to easily prove the following comparison between multiplicity and F-threshold:
 \begin{Prop}
For $f \in \mf{m} \subset  k[x_,\dots, x_N]$, 
$$
\frac{N}{mult_{\mf{m}}(f) }\,\, \geq \,\,FT_{\mf{m}}(f)\,\, \geq\,\, \frac{1}{mult_{\mf{m}}(f)}.
$$
\end{Prop}

\begin{proof}
Since $\mf{m}$ is generated by $N$ elements, 
we have the inclusions
$$
\mf{m}^{Np^e} \subset \mf{m}^{[p^e]} \subset \mf{m}^{p^e}
$$
for all $e$. So we also obviously have inclusions of sets
$$
\Big\{\frac{a}{p^e} \in \mathbb Z\Big[\frac{1}{p}\Big] \,\, \ \big|\ f^a \in \mf{m}^{Np^e}\Big\} \subset 
\Big\{\frac{a}{p^e} \in \mathbb Z\Big[\frac{1}{p}\Big] \,\, \ \big|\ f^a \in \mf{m}^{[p^e]}\Big\} \subset
\Big\{\frac{a}{p^e} \in \mathbb Z\Big[\frac{1}{p}\Big] \,\, \ \big|\ f^a \in \mf{m}^{p^e}\Big\}.
$$
Taking the infimum, we have
$$
\inf \Big\{\frac{a}{p^e} \in \mathbb Z\Big[\frac{1}{p}\Big] \,\, \ \big|\ f^a \in \mf{m}^{Np^e}\Big\} \geq 
\inf \Big\{\frac{a}{p^e} \in \mathbb Z\Big[\frac{1}{p}\Big] \,\, \ \big|\ f^a \in \mf{m}^{[p^e]}\Big\} \geq 
\inf \Big\{\frac{a}{p^e} \in \mathbb Z\Big[\frac{1}{p}\Big] \,\, \ \big|\ f^a \in \mf{m}^{p^e}\Big\}.
$$
Since the infimum on the left can be interpreted as $N $ times $ \inf \Big\{\frac{a}{Np^e} \in \mathbb Z\Big[\frac{1}{p}\Big] \,\, \ \big|\ f^a \in \mf{m}^{Np^e}\Big\}$, 
the result is proved.
\end{proof}

\begin{Rmk}
A deeper inequality is proven in \cite[Prop 4.5]{TW}: $mult(f) \geq \frac{N^N}{(FT(f))^N}$. This in turn, is a ``characteristic $p$ analog" of a corresponding statement in characteristic zero about log canonical threshold \cite[Thm 1]{dEM}.
\end{Rmk}

\subsection{Computing $F$-thresholds.}

To get a feeling how to compute $F$-thresholds, we begin the computation of Example \ref{cuspide}, relegating the details to \cite{Her}.  First recall that $\F_p[x, y]$ is a free module over $\F_p[x^{p^e}, y^{p^e}]$ with basis $\{x^{a_1}y^{a_2}\}_{0\leq a_1,a_2\leq p^e-1}$. By definition,
 $$
 FT(f)=\sup\Big\{\frac{a}{p^e}\ \big|\ f^a=\sum r_A^{p^e} x^{a_1}y^{a_2} \hbox{has some coefficient } \,\, r_A^{p^e}\not\in (x^{p^e}, y^{p^e})\Big\}.
 $$ 
 
 For each $a$, we  expand using the binomial theorem to get
 $$
 (x^2+y^3)^a=\sum_{i=0}^a\binom{a}{i}x^{2i}y^{3(a-i)}.
 $$
 Note that none of the terms in this expression can cancel, since each has a unique bi-degree in $x$ and $y$, unless its corresponding binomial coefficient is zero.
 So, thinking over $\F[x^{p^e}, y^{p^e}]$, we see that 
$FT(f)\geq \frac{a}{p^e}$ if and only if there is an index  $i$ such that 
$$2i<p^e,\quad 3(a-i)<p^e\quad\hbox{and }\ \binom{a}{i}\not\equiv 0\pmod{p}.\qquad (*)$$

Note that if  $2a<p^e$, then the index $i=a$ fulfills the three conditions in $(*)$. This, in particular, implies that $FT(f)\geq \frac{1}{2}$, independent of $p$. If the characteristic is 2, it is easy to see immediately that $FT(f) = \frac{1}{2}$.

On the other hand, if $\frac{a}{p^e} \geq \frac{5}{6}$, then condition $(*)$ is never satisfied, so that $FT(f) \leq \frac{5}{6}$, independent of $p$. Indeed, in this case, 
either $2i \geq p^e$ or $3(a-i) \geq p^e$. For otherwise, we have both
$$
i \leq \frac{p^e-1}{2} \,\,\,\,\,\,\,\,\,\,\, {\rm and} \,\,\,\,\,\,\,\,\,\, \,\,\,\,\,(a-i) \leq \frac{p^e-1}{3},
$$
in which case, adding them, we have  that
$
a \leq \frac{p^e-1}{2} + \frac{p^e-1}{3}.
$ This implies that
$\frac{a}{p^e} \leq \frac{5}{6} - \frac{5}{6p^e}$, a contradiction. 

For the exact computation of the $F$-threshold, it remains to analyze the binomial coefficients in the critical  terms in which both $2i<p^e$ and $3(a-i)<p^e$. These are the terms indexed by  $i$ satisfying $a - \frac{p^e}{3} <  i  <\frac{p^e}{2}$. For this, it is crucial to understand the behavior of  binomial coefficients modulo $p$.  One of the main tools is the following theorem:

\begin{Thm}{\rm (}Lucas, \cite{Lu}{\rm ).}
Fix non-negative integers $m \geq n\in\N$ and a prime number $p$. Write $m$ and $n$ in their base $p$ expansions:  
$m=\sum_{j=0}^r m_jp^j$ and $n=\sum_{j=0}^r n_jp^j$.  Then, modulo $p$, 
$$\binom{m}{n} \equiv \binom{m_0}{n_0}\binom{m_1}{n_1} \cdots \binom{m_r}{n_r},
$$
where we interpret $\binom{a}{b}$ as zero if $ a < b$.
In particular, $\binom{m}{n}$ is non-zero mod $p$ if and only if  $m_j\geq n_j$ for all $j=1,\dots,r$.
\end{Thm}
\end{ex}

Thus, to compute the $F$-threshold of $x^2+ y^3$, it is helpful to write $a$ in its base $p$ expansion, and then try to understand, using Lucas's theorem,  whether there exist values of $i$ in the critical range for which $\binom{a}{i}$ is not zero. For example, if $p \equiv 1 \mod 6$, then $5p^e = 5 \mod 6$, so the number $a = \frac{5p^e - 5}{6}$ is an integer.  With this choice of $a$, it is not hard to check that  the  conditions $(*)$ are satisfied for the index $i = \frac{p^e - 1}{2}$. This shows that when $p   \equiv 1 \mod 6,$ then the $F$-threshold of $x^2 + y^3$ is at least $\frac{a}{p^e} = \frac{5}{6} - \frac{5}{p^e}$ for all  $e$. It follows that for  $p   \equiv 1 \mod 6,$ the $F$-threshold of $x^2 + y^3$ is exactly $\frac{5}{6}$.
 The details, as well as the computation for other $p$,  are carried out in \cite[Example 4.3]{MTW} or \cite[Example 8.2]{Her}.

\subsection{Comparison of $F$-thresholds and log canonical thresholds.}
Once $F$-thresholds and log canonical thresholds are defined, it is natural to compare them when it is possible. This is the case when $f\in\Z[x_1,\ldots,x_n]$. We can view $f$ as a polynomial over $\C$, and compute its log canonical threshold. After reduction modulo $p$,  we can calculate the $F$-threshold of $f \mod p$ in $\F_p[x_1, \dots, x_n]$.

\vspace{0.3cm}

\noindent {\bf Question:} For which values of $p$ is $lct(f)=FT(f \mod p)$? What happens when $p\gg 0$?

\vspace{0.3cm}

The following theorem provides a partial answer:

\begin{Thm}\label{reduce}
Fix $f\in\Z[x_1,\dots,x_n]$. Then,
\begin{itemize}
\item $FT(f \mod p)\leq lct(f)$ for all $p \gg 0$ prime.
\item $\lim_{p\rightarrow\infty} FT(f \mod p)=lct(f)$.
\end{itemize}
\end{Thm}

The proof of this theorem is the culmination mainly of the work of the Japanese school of tight closure, who generalized the theory of tight closure to the case of pairs. The first important step was the work of Hara and Watanabe in \cite{HW} Theorem 3.3,  with Theorem \ref{reduce} essentially following from  Theorem 6.8 in the paper \cite{HY} of Hara and Yoshida; the proof there in turn generalizes the proofs given in \cite{Hara} and \cite{Sm00} in the non-relative case to pairs.

\noindent{\bf Open Problem:} Are there infinitely many primes $p$ for which $FT(f\mod p)=lct(f)$?

A positive answer to this Question would settle a long standing conjecture in $F$-singularities: every log canonical pair $(X, D)$ where $X$ is smooth  (over $\mathbb C$) is of  ``$F$-pure type." Daniel Hernandez shows that this is the case for a ``very general"  polynomial $f$ in $\C[x_1, \dots, x_n]$  \cite{Her0}. 
Versions of this question have been around since the early eighties, for example, as early as Rich Fedder's thesis in 1983 \cite{Fed}, when similarities between Hochster and Robert's notion of ``F-purity" and rational singularities began to emerge. Watanabe pointed out that the log-canonicity ought to correspond to F-purity, circulating a preprint in the late eighties proving that ``F-pure implies log canonical"  for rings.  He did not publish this result for several years, when together with Hara, they introduced a notion of F-purity for pairs and proved the corresponding result for a pair $(X, D)$. The field of  ``F-singularities of pairs," including the F-pure threshold, developed rapidly in Japan, with numerous papers of Watanabe, Hara, Takagi, Yoshida, and others filling in the theory. The question in the form stated here may have first appeared in print in \cite{MTW}.

\vspace{0.2cm}

\subsection{Test ideals and $F$-thresholds}\label{lecture3}
We wish to construct a family of ideals of a polynomial $f \in \F_p[x_1, \dots, x_n] = R$, say $\tau(f^c)$, called test ideals,  which are analogous to the multiplier ideals.

We first restate the definition of $F$-threshold yet one more time, in a way that will make the definition of the test ideals very natural.
\begin{defi}
The $F$-threshold of $f \in \F_p[x_1, \dots, x_n]$ at ${\mf{m}}$ is 
$$
FT_{\mf{m}}(f) = \sup\Big\{c=\frac{a}{p^e} \in \mathbb Z\Big[\frac{1}{p}\Big] \,\,  \big|\ \exists \varphi\in\Hom_{R}(R^{1/p^e}, R), \ \varphi(f^{c}) \notin 
 \mf{m} \Big\}.$$
\end{defi}

To see that this is equivalent to the previous definitions, we apply the following simple lemma in the case where $J \ = \mf{m}$:

\begin{Lem}\label{simple} Consider $f \in  R = \F_p[x_1, \dots, x_n]$ and $c = \frac{a}{p^e}$. 
Fix any ideal $J \subset R$. 
Then  $f^c \notin J R^{\frac{1}{p^e}}$ if and only if there is an $R$-linear map $R^{\frac{1}{p^e}}\overset{\phi}\longrightarrow R$ sending $f^c$ to an element not in $J$. 
\end{Lem}

\begin{proof}
Suppose that  $f^c \notin J R^{\frac{1}{p^e}}$. Then in writing $f^c$ uniquely in some basis for $R^{\frac{1}{p^e}}$ as in expression (\ref{FTex}), there is some coefficient $r_A \notin J$. If we let $\phi$ be the projection onto this direct summand, we have a map $\phi$ satisfying the required conditions. Conversely,  if $f^c \in J R^{\frac{1}{p^e}}$, then the $R$-linearity of $\phi$ forces   $\phi(f^c) \in  J$ for any $\phi \in \Hom_{R}(R^{1/p^e}, R)$. So every such  $\phi$ must  send $f^c$ to an element  in $J$.
\end{proof}

With this definition of $F$-threshold in mind, it is quite natural to define test ideals, at least for certain $c$.
First note that for each $f \in \F_p[x_1, \dots, x_n]$ and  each $c\in\Z[\frac{1}{p}]$,
we have a natural $R$-module map:
$$\xymatrix@R=0pc@C=1pc{
\Hom(R^{1/p^e},R)\ar[r]& R \\
\ \ \ \phi\ar@{|->}[r]& \phi(f^c),
}$$
where $e$ is chosen so that $c = \frac{a}{p^e}$ for some natural number $a$.
The test ideal is the image of this map:
\begin{defi}\label{test} Let $f \in \F_p[x_1, \dots, x_n]$ and  $c =\frac{a}{p^e}\in\Z[\frac{1}{p}]$.
The \emph{test ideal} $\tau(f^c)$ is the ideal 
$$\tau(f^c)=\hbox{im}[\Hom(R^{1/p^e},R)\lra R,\ \hbox{defined by evaluation at $f^c$} ].$$
\end{defi}

In practical terms, if we write $f^c = f^{a/p^e}$ uniquely as in Expression (\ref{FTex}), then 
$\tau(f^c)$ is generated by the coefficients $r_A$ which appear in this expression. Note that this 
 is independent of the way we write $c$. Indeed, if we instead think of $c$ as $\frac{ap}{p^{e+1}}$, 
then the expression for $f^c$ becomes
$$
f^{c} = \sum_A r_A x^{A/p^e} = \sum r_A x^{pA/p^{e+1}}
$$
which is a valid expression for $f^c$ as an element of the free $R$-module $R^{1/p^{e+1}}$ since the monomials $ x^{pA/p^{e+1}}
$ are still part of free basis for $R^{1/p^{e+1}}$.

\begin{Rmk}\label{simplecor}
The test ideal $\tau(f^{a/p^e})$ is the {\it smallest\/} ideal $J$ such that  $f^{a/p^e} \in JR^{1/p^e}$.
Indeed,  Lemma \ref{simple} can be reinterpreted as saying that if there is some ideal $J$ of $R$ such that $f^{a/p^e} \in JR^{1/p^e}$, then $\tau(f^{a/p^e}) \subset J.$

 \end{Rmk}

\medskip
Although we have not yet  defined  test ideals with arbitrary real exponents $c$, let us pause and see whether, at least for   $c \in \mathbb Z[1/p]$, we have a collection of ideals 
$\{\tau(f^c)\}$  satisfying the desired basic properties analogous to Proposition \ref{basic}.
That is, we want
\begin{enumerate}
\item $\tau(f^{c}) $ is the unit ideal for sufficiently small positive $c$;
    \item If $c > c'$, then $ \tau(f^{c'}) \supseteq  \tau(f^{c});$
     \item  The $F$-threshold of $f$ is 
 $$FT_{\mf{m}}(f)=\sup \{ c\ | \ \tau(f^c)=(1) \}.$$
      \item  If $\varepsilon >0$ small enough, then $\tau(f^c)= \tau(f^{c+\varepsilon})$.   
  \item  There exist certain $c$ such that $\tau(f^{c - \varepsilon})\supsetneq \tau(f^c)$ for all positive $\varepsilon$.
\end{enumerate}

The first property is easy. Indeed,  for fixed $a$ consider  $f^{a/p^e} \in R^{1/p^e}$ as $e$ gets very large.  If  $f^{a/p^e} \in \mf{m}R^{1/p^e},$  then  $f^a \in m^{[p^e]} \subset m^{p^e}$ in $R$. But this can not be the case  for all $e$. The third property follows immediately from the second, which is also 
quite straightforward. Note that since the test ideal $\tau(f^c)$ is independent of the way we represent $c$ as a quotient of two integers whose denominator is a power of $p$, we may assume that $c = \frac{a}{p^e}$ and $c' = \frac{b}{p^e}$ have a common denominator. Then:

\begin{Lem}\label{vouyavess}
For $\frac{a}{p^e} \geq \frac{b}{p^e}$, we have
$$ \tau( f^{a/ p^e}) \subseteq \tau(f^{b/ p^e}).$$
\end{Lem}
\begin{proof}
 Suppose that $s \in \tau(f^{a/p^e})$. This means there is an $R$-linear map
 $\psi:R^{1/p^{e}}\lra R$ so that $\psi(f^{a/p^e})=s$. Precomposing this with the $R$-linear map 
 $R^{1/p^e} \overset{\mu}\longrightarrow R^{1/p^e}$ given by multiplication by $f^{(a-b)/p^e}$, we have an $R$-linear map
$$\xymatrix@R=0pc@C=2pc{
 R^{1/p^e} \ar[r]^{\mu} & R^{1/p^e}\ar[r]^{\psi} & R\\
 f^{b/p^e} \ar@{|->}[r]& f^{(a-b)/p^e} f^{b/p^e}\ar@{|->}[r] & \varphi(f^{a/p^e}) = s,
}$$
showing that $s \in \tau(f^{b/p^e})$ as well.
\end{proof}

The fourth property is also quite simple to prove: 

\begin{Lem} \label{4} Fix $c = \frac{a}{p^e}$. Then for all $n $ sufficiently large (how large could depend on $c$), 
$$
\tau(f^c) = \tau(f^{c+\frac{1}{p^n}}).
$$
\end{Lem}

\begin{proof} 
Fix $c = \frac{a}{p^e}$ and let $s\in
\tau(f^c). $ This means that there exists an $R$-linear map 
$$R^{1/p^e} \overset{\phi}\rightarrow R \,\,\,\,\,\,\,\,\,\, s.t. \,\,\,\,\,\,f^{a/p^e} \mapsto s.$$ Take $n$ sufficiently large so that $f^{1/p^{n-e}}$ is part of a free basis for $R^{1/p^{n-e}}$. 
 Projection onto the submodule spanned by $f^{1/p^{n-e}}$ is  an $R$-linear map 
 $R^{1/p^{n-e}}  \rightarrow R$    sending  $f^{1/p^{n-e}}$ to  $1.$ Taking the $p^e$-th roots, we have  an $R^{1/p^{e}}$-linear map 
  $$R^{1/p^{n}}  \overset{\psi}\rightarrow R^{1/p^e} \,\,\,\,\,\,s.t. \,\,\,\,\,\, f^{1/p^n} \mapsto 1.$$
   In particular 
$\psi(f^{\frac{a}{p^e} + \frac{1}{p^n}})  = f^{\frac{a}{p^e}}.$
 Composing, we get an   $R$-linear map 
$$\xymatrix@R=0pc@C=2pc{
 R^{1/p^n} \ar[r]^{\psi} & R^{1/p^e}\ar[r]^{\phi} & R\\
 f^{a/p^e + 1/p^n} \ar@{|->}[r]& f^{a/p^e} \ar@{|->}[r] & \varphi(f^{a/p^e}) = s.
}$$
This shows that $s \in \tau(f^{c+\frac{1}{p^n}}),$ as desired.
\end{proof}

\medskip The fifth property, however, is {\it not true}  if we restrict attention to $c$ that are rational numbers whose denominator is a power of $p$. To get property (5), we need to define 
define test ideals also for arbitrary positive real numbers $c; $  if we can do so in such a way that the first four properties are satisfied for all real numbers, then the
 completeness property of the real numbers will automatically grant (5).

Lemma \ref{vouyavess} encourages us to define $\tau(f^c)$ for any positive real $c$  by approximating $c$  by a sequence of numbers  $\{c_n\}_{n \in \N}  \in \Z[\frac{1}{p}]$ converging to $c$ {\it from above\/} and taking advantage of the Noetherian property of the ring. That
 is, we take any monotone decreasing sequence of numbers in $\Z[\frac{1}{p}]$
$$ 
c_1 > c_2 > c_3 \dots
$$
 converging to $c$.
 There is a corresponding increasing sequence of test ideals:
$$
\tau(f^{c_1}) \subseteq \tau( f^{c_2}) \subseteq  \tau( f^{c_3}) \cdots .$$
Because $R$ is Noetherian, this chain of ideals must stablize.  
Since any other strictly  decreasing sequence converging to $c$ is cofinal with this one (meaning that, if $\{c_i'\}$ is some other sequence, then  for all $i$, there exists $j$ such that $c_i > c_j'$ and vice versa), it is easy to check that the stable ideal is independent of the choice of 
approximating sequence.
So we have the following definition:

\begin{defi}\label{testideal}
For any $c \in \R_+$ the \emph{test ideal} is defined as
$$ \tau (f^c):= \bigcup_{n \geq 0} \tau(f^{c_n}),$$
where $\{c_n\}_{n \in \N}$ is any decreasing sequence of rational numbers in $\Z[1/p]$ approaching $c$. In particular, 
$$ \tau (f^c):= \bigcup_{n \geq 0} \tau(f^{\lceil cp^n \rceil / p^n}),$$
or equivalently as  $\tau(f^{\lceil cp^n \rceil / p^n})$ for $n \gg 0$. \end{defi}

There is one slight ambiguity to address: If the real number $c$ {\it happens} to be a rational number whose denominator is a power of $p$, then we have already defined $\tau(f^c)$ in Definition \ref{test}. Do the two definitions produce the same ideal in this case? That is, we 
 need to check that, if we had instead approximated $c$ by a sequence $\{c_n\}_{n\in \N} \in \Z[1/p]$ converging to $a/p^e$ from above,  then the ideals $\tau(f^{c_n}) $ stabilize to $\tau(f^{a/p^e})$.
 But this is essentially the content of Lemma \ref{4}. So test ideals are well-defined for any positive real number.

\medskip
As before with multiplier ideals, the following  follows easily from the definition:
   \begin{Prop} Fix a polynomial $f$ and view its test ideals $ \tau(f^c)$ as a family of ideals varying with a positive real parameter $c$. Then the following properties hold:
   \begin{enumerate}
    \item For $c \in \R_+$ sufficiently small, $ \tau(f^c)$  is the unit ideal. 
    \item If $c > c'$, then $ \tau(f^{c}) \subseteq  \tau(f^{c'}).$

   \item  The $F$-threshold of $f$ is 
 $$FT(f)=\sup \{ c\ | \ \tau(f^c)=(1) \}.$$
  \item For each fixed $c$, we have  $\tau(f^c)= \tau(f^{c+\varepsilon})$ for sufficiently small  positive $\varepsilon $ (how small is small enough depends on $c$).   
 \item  There exist certain $c\in \R_+$ such that $\tau(f^{c - \varepsilon})\supsetneq \tau(f^c)$ for all positive $\varepsilon$.
\end{enumerate}
\end{Prop}

The proposition is summarized by the following diagram of the $c$-axis, which shows intervals where the test ideal remains constant:
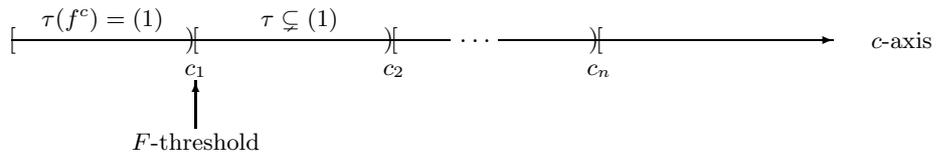
\setlength{\unitlength}{0.9mm}
\begin{center}
\begin{picture}(120,23)
\put(0,15){\makebox(0,0){$[$}}
\put(0,15){\line(1,0){64}}
\put(68,15){\makebox(0,0){$\dots$}}
\put(71,15){\vector(1,0){49}}
\put(130,15){\makebox(0,0){{\small $c$-axis}}}

\put(26,15){\makebox(0,0){$)$}}
\put(26.8,15){\makebox(0,0){$[$}}
\put(26.8,10.5){\makebox(0,0){\small $c_1$}}
\put(13.2,18){\makebox(0,0){\small${\tau(f^{c})=(1)}$}}

\put(55,15){\makebox(0,0){$)$}}
\put(55.8,15){\makebox(0,0){$[$}}
\put(55.8,10.5){\makebox(0,0){\small $c_2$}}
\put(42,18){\makebox(0,0){\small $\tau\subsetneq(1)$}}

\put(85,15){\makebox(0,0){$)$}}
\put(85.8,15){\makebox(0,0){$[$}}
\put(85.8,10.5){\makebox(0,0){\small $c_n$}}

\put(26.8,2){\vector(0,1){7}}
\put(26.8,0){\makebox(0,0){\small $F$-threshold}}

\end{picture}
\end{center}

This leads naturally to the $F$-jumping numbers:

\begin{defi}
The \emph{$F$-jumping numbers} of $f$ are the real numbers $c_i \in \R_+$ for which $\tau(f^{c_i})\neq\tau(f^{c_i-\varepsilon})$, for every $\varepsilon>0$.
\end{defi}

It is not hard to see that test ideals and F-jumping numbers enjoy many of the same properties as do the multiplier ideals and jumping numbers defined in characteristic zero. First we have an analog of the Brian\c con-Skoda theorem.{\footnote{Starting in \cite{Laz}, the name of this theorem, 
which belongs to a collection of inter-related theorems comparing powers of an ideal to its integral closure, has been sometimes shortened to ``Skoda's theorem." We follow here the tradition in commutative algebra to include Brian\c con's name.}}

\begin{Prop}{\rm(}\cite[Proposition 2.25]{BMS1}{\rm).}\label{BS2}
Let $f$ be a polynomial in $\F_p[x_1, \dots, x_n]$. Then for every $c \in \R_+$, we have 
$$\tau(f^{c+1})=(f) \cdot \tau(f^c).$$ 
In particular, a positive real number $c$ is an $F$-jumping number if and only if $c + 1$ is an $F$-jumping number.
\end{Prop}

\begin{proof}
Without loss of generality,  we may replace both $c$ and $c+1$ by  rational  numbers in  $\Z[\frac{1}{p}]$ approximating each from above. Thus we may assume $c =  a/p^e$ for some $a, e \in \N$.

For any $R$-linear map $R^{1/p^e}\overset{\varphi}\lra R$, it is clear that 
$$\varphi(f^{(a/p^e)+1}) = f \varphi(f^{(a/p^e)}),$$ 
since $f \in R$. It immediately follows that $\tau(f^{c+1})=(f) \cdot \tau(f^c).$
\end{proof}

Like the jumping numbers in characteristic zero, the $F$-jumping numbers are discrete and rational. Interestingly, the proof of the characteristic zero statement follows trivially from the (algebro-geometric) definition of multiplier ideals, while the characteristic $p$ proof took some time to find.

\begin{Thm}{\rm(}\cite[Theorem 3.1]{BMS1}{\rm).}\label{rat} The $F$-jumping numbers of a polynomial  are discrete and rational.
\end{Thm}

The proof takes advantage of an additional symmetry  the F-jumping numbers enjoy for  which there is no analog in characteristic zero:

\begin{Lem}{\rm(}\cite[Proposition 3.4]{BMS1}{\rm).}\label{dumpling}
Let $f$ be a polynomial in $\F_p[x_1, \dots, x_n]$. If $c$ is a jumping number for $f$, then also $pc$ is a jumping number for $f$.
\end{Lem}

\begin{proof}
Let $c$ be a jumping number, that is, suppose $\tau(f^c)\subsetneq \tau(f^{c-\varepsilon})$ for all $\epsilon>0$.   For any $ a/p^e < c \leq b/p^e$ where $a, b$ are positive integers, we thus have
$$
\tau(f^{\frac{b}{p^e}}) \subset \tau(f^{c})  \subsetneq \tau(f^{\frac{a}{p^e}}), 
$$
and the first inclusion is an equality when $\frac{b}{p^e}$ is sufficiently close to $c$.
For such close $\frac{b}{p^e}$, write
$$f^{b/p^e} = \sum_B r_B \x^{B/p^e},$$
so  that $\tau(f^{b/p^e}) $ is generated by the coefficients $r_B$.
Raising to the power $p$ gives that 
$$
f^{b/p^{e-1}} = \sum_B r_B^p \x^{B/p^{e-1}}, 
$$ which means, by Lemma \ref{simple}, that every $R$-linear map $R^{1/p^{e-1}} \overset{\phi}\rightarrow R$ sends $f^{b/p^{e-1}}$ to something in $ \langle r_B^p \rangle $.
In other words, 
\begin{equation}\label{5}
 \tau(f^{b/p^{e-1}}) \subseteq \langle r_B^p \rangle.
 \end{equation} 

 In contrast,  since $ \tau(f^{b/p^e}) \subsetneq \tau(f^{a/p^e}), $ we know  that $ f^{a/p^e} \not\in  \langle r_B \rangle R^{1/p^e}.  $ 
 Raising to $p$-th powers again, we have that $f^{a/p^{e-1}} \not\in  \langle r_B^p \rangle  R^{1/p^{e-1}}$. But then Equation \ref{5} forces 
   $f^{a/p^{e-1}} \not\in  \tau(f^{\frac{b}{p^{e-1}}}) R^{1/p^{e-1}}$. 
 By Lemma \ref{simple}, it follows that there is   an $R$-module homomorphism $R^{1/p^e} \overset{\phi}\rightarrow R$ such that 
$\phi(f^{a/p^{e-1}}) $ is not in $  \tau(f^{\frac{b}{p^{e-1}}})$.
But this exactly means that there is an element of $\tau(f^{a/p^{e-1}}) $ that is not in $  \tau(f^{\frac{b}{p^{e-1}}})$.
 Letting $a/p^e, b/p^e$  go to $c$, we get that  $ pc$  is a jumping number as well as $c.
 $
\end{proof}

\begin{proof}[Proof of Theorem \ref{rat}]
To prove discreteness, we fix an $f$ of degree $d$, and $c =\frac{a}{p^e}$. We claim that $\tau(f^c)$ is generated by elements of degree smaller or equal to $\lfloor cd \rfloor$. Indeed,  $\tau(f^c)$ is generated by the coefficients $r_A$ appearing in $f^c=\sum r_A \x^{a/p^e} \in R^{1/p^e}.$  Since $f^c$ has degree $dc$, then $r_A \in R$ has degree $\leq \lfloor cd \rfloor$. This proves the claim. 

Now assume that the $F$-jumping numbers of $f$, say   $\alpha_1 < \alpha_2 < \cdots$, were clustering to some $\alpha$.  Without loss of generality, each of the $\alpha_i$ can be assumed in $\Z[\frac{1}{p}]$. By definition of $F$-jumping number, 
\begin{equation}\label{marseille}
\tau(f^{\alpha_1}) \supsetneq \tau(f^{\alpha_2}) \supsetneq \cdots.
\end{equation} The previous claim ensures that
each $\tau(f^{\alpha_i})$ is  generated in degree  $\leq   \lfloor d \alpha_i \rfloor  \leq \lfloor d \alpha \rfloor = D$. Now intersect each of these test ideals with the finite dimensional vector space $V \subseteq \F_p[x_1, \ldots, x_n]$ consisting of polynomials of degree $\leq D$. The sequence 
$$\tau( f^{\alpha_1}) \cap V \supsetneq \tau( f^{\alpha_2}) \cap V \supsetneq \cdots$$
stabilizes, since $V$ is finite dimensional. Hence (\ref{marseille}) also stabilizes, which is a contradiction to clustering of the $F$-jumping numbers. 

To prove the rationality of $F$-jumping numbers, let  $c \in \R $ be an $F$-jumping number. Then for all $e\in \N$, the real numbers  $p^ec$ are also $F$-jumping numbers. One can write $p^e c=\lfloor p^e c \rfloor + \{ p^e c \}$, where the fractional part $\{ p^e c\}$ is also an $F$-jumping number by Lemma \ref{dumpling}. By the discreteness of the $F$-jumping numbers it follows that $\{ p^e c \}=\{ p^{e'}c\}$ for some $e$ and $e'$ in $\N$, and hence $p^e c - p^{e'}c=m \in \Z$. Thus $c= \frac{m}{p^e- p^{e'}}$ is rational. 
\end{proof}

\begin{Rmk} Test ideals and multiplier ideals can be defined not just for one polynomial, but for any ideal $\mathfrak a$ in any polynomial ring, and even for sheaves of ideals on (certain) singular ambient schemes. While not much more complicated that what we have introduced here, 
we refer the interested reader to the literature for this generalization. Many properties of multiplier ideals 
(see \cite[Chapter 9]{Laz}) can be directly proven, or adapted, to test ideals. For example, the fact that the $F$-jumping numbers are discrete and rational holds more generally---essentially for all ideals in normal $\Q$-Gorenstein 
ambient schemes \cite{BSTZ}. 
 In addition to the few properties discussed here, other properties of multiplier ideals  that have  analogs for test ideals include the ``restriction theorem," the ``subadditivity theorem," the ``summation theorem" \cite{HY} \cite{T},  and the behavior of test ideals under finite morphisms \cite{ST2}.
 Interestingly, some of the more difficult properties to prove in characteristic zero turn out to be extraordinarily simple in characteristic $p$. For example, in characteristic zero,  the proof of the Brian\c con-Skoda theorem (for ideals that are not necessarily principal) uses  the ``local vanishing theorem" (see \cite{Laz}).  Although this vanishing theorem fails in characteristic $p$, the characteristic $p$ analog of the Brian\c con-Skoda Theorem (that is, Theorem \ref{BS2} for non-principal ideals) is none-the-less true and in fact quite simple to prove immediately from the definition. On the other hand, some properties of multiplier ideals that follow immediately from the definition in terms of resolution of singularities turn out to be {\it false\/} for test ideals. For example,  while multiplier ideals are easily seen to be integrally closed,  test ideals are not. In fact, {\it every\/} ideal in a polynomial ring is the test ideal $\tau(\mathfrak a^{\lambda})$ for some ideal $\mathfrak a$ and some positive $\lambda \in \mathbb R$, as shown in \cite{MY}.
  \end{Rmk}

\vspace{0.1cm}

\subsection{\bf An interpretation of $F$-thresholds and test ideals using differential operators.}

Our definition of  $F$-threshold can be viewed as a measure of singularities using differential operators.  The point is that a differential operator on a ring $R$ of characteristic $p$ is precisely the same as a $R^{p^e}$-linear map.

Differential operators can be defined quite generally. Let $A$ be any base ring, and $R$ a commutative $A$-algebra. Grothendieck defined the ring of $A$-linear differential operators of $R$ using a purely algebraic approach  (see \cite{EGAIV4}), which in the case where $A = k$ is a field and $R$ is a polynomial ring over $k$ results in the ``usual" differential operators.

\begin{defi} Let $R$ be a commutative $A$-algebra. 
The ring $D_A(R)$ of $A$-linear differential operators is the subring of the (non-commutative) ring 
$\End_{A}(R)$ obtained as the union of  the $A$-submodules of differential operators $D_A^n(R)$
of order less than or equal to $n$, where $D^{n}_A(R)$ is defined inductively as follows:
First the zero-th order operators $D^0_A(R)$ are the elements $r \in R$ interpreted as the
$A$-linear endomorphisms ``multiplication by $r$"; that is, $r:R \lra R$ sends $x\mapsto rx$. 
Then, for $n >0$,
$$
D^{n}_A(R) := \{\partial \in \End_A(R) \,\,\, |\,\,\, [r, \partial] \in D^{n-1}_A(R) \,\,\,\,{\rm for \,\,\,all \/}\,\,\,\, r\in R\}, 
$$
where $[r, \partial]$ is the usual Lie bracket of operators, i.e., $[r, \partial]=r \circ \partial - \partial \circ r$. 
\end{defi}

\begin{ex}
The elements of $D^1_A(R)$ consist of the endomorphisms of the form $r + d$ where $d$ is an $A$-linear derivation of $R$ and $r \in R$.
\end{ex}
\begin{ex}\label{char0diff} If $k$ has characteristic zero and $R$ is the polynomial ring $k[x_1, \dots, x_n]$, then $D_k(R)$ is the Weil algebra $k[x_1, \dots, x_n, \partial_1, \dots, \partial_n]$ where each $\partial_i$ denotes the derivation $\frac{\partial}{\partial x_i}$.  This is the non-commutative subalgebra of $\End_k(R)$ generated by the $\frac{\partial}{\partial x_i}$ and the multiplication by $x_j$. 
\end{ex}

\begin{ex}\label{charpdiff}
In characteristic $p$, the differential operators on $k[x_1,\dots, x_n]$ are essentially ``the same" as in Example \ref{char0diff} as $k$-vector spaces, but not as rings. For example, 
if $k$ has characteristic $p$, the operator  
$$\Big(\frac{\partial}{\partial x_i}\Big)^p = \underset{p \,\, times}{\underbrace{{\Big(\frac{\partial}{\partial x_i} \circ  \cdots \circ \frac{\partial}{\partial x_i}\Big)}}}
$$
obtained by composing the first order operator $\frac{\partial}{\partial x_i}$ with itself $p$-times is the {\it zero operator.\/} None-the-less, there is a differential operator
$$\frac{1}{p!}\frac{\partial^p}{\partial x_i^p}$$ sending $x_i^p$ to $1$,
which is  {\it not\/} the composition of lower order operators but which essentially has the same effect as the corresponding composition in characteristic zero.  
In particular, a $k$-basis for  $D_k(k[x_1,\dots, x_n])$, where $k$ has characteristic $p$, is
$$
\Big\{  x_1^{j_1}\cdots x_n^{j_n}, \,\, \frac{1}{i_1!}\frac{\partial^{i_1}}{\partial x_1^{i_1}} \circ \frac{1}{i_2!}\frac{\partial^{i_2}}{\partial x_2^{i_2}} \circ \cdots \circ 
\frac{1}{i_n!}\frac{\partial^{i_n}}{\partial x_n^{i_n}} \Big\}
$$
as $j_{\ell}$ and  $i_{\ell}$ range over all non-negative integers. 
In characteristic $p$, $D_k(k[x_1,\dots, x_n])$ is not finitely generated.
\end{ex}
\medskip
The following alternate interpretation of differential operators in characteristic $p$ ties into the definition of F-threshold.
\begin{Prop}
Let $R$ be any ring of prime characteristic $p$ such that the Frobenius map is finite (main case for us: 
$R = \F_p[x_1, \dots, x_n]$). Then an $\F_p$-linear map $R\overset{d}\lra R$ is a differential operator if and only if it is linear over some subring of $p^e$-th powers. That is,
$$
D_{\F_p}(R) = \bigcup_{e\in \N} \End_{R^{p^e}}(R).
$$
\end{Prop}
\begin{proof} This is not a difficult fact to prove. It is first due to  \cite{Yek}; or see \cite{SV} for a detailed proof in this generality.
\end{proof}

This gives us an alternate filtration of $D_{\F_p}(R)$ by ``Frobenius order:"
$
D_{\F_p}(R)
$
is the union of the chain
$$
\End_{R}(R) = R \subset \End_{R^p}(R) \subset  \End_{R^{p^2}}(R) \subset \End_{R^{p^3}}(R) \subset \dots
$$
Alternatively, by taking $p^e$-th roots, we can interpret the ring of differential operators as the union
$$
\End_{R}(R) = R \subset \End_{R}(R^{\frac{1}{p}}) \subset  \End_{R}(R^{\frac{1}{p^2}}) \subset 
\End_{R}(R^{\frac{1}{p^3}})\subset \dots
$$

Using this filtration, we can give an alternative definition of test ideals and $F$-threshold in terms of differential operators:
\begin{defi}
The \emph{$F$-threshold} of $f \in  \F_p[x_1,\ldots,x_n]$ at the maximal ideal $\mf{m} $ is defined as
$$
FT_{\mf{m}}(f)=\sup\Big\{c=\frac{a}{p^e} \in \mathbb Z\Big[\frac{1}{p}\Big] \,\, \ \big|\ \exists\ \partial \in \End_{R}(R^{1/p^e}) \hbox{\,\, such that \ } \partial(f^{a/p^e}) \not\in\mf{m}R^{1/p^e}\Big\}.
$$
\end{defi}
Note that this interprets the F-threshold as very much like the multiplicity: it is defined as the maximal (Frobenius) order of a differential operator which, when applied to (a power of) $f$, we get a non-vanishing function. However, here the operators are filtered using Frobenius. 

Similarly, the test ideal is the image of $f^{a/p^e}$ over all differential operators of $R^{1/p^e}$ which have image in $R$.


\begin{Rmk}{\bf Historical Remarks and Further Work.}
The $F$-threshold was first defined by Shunsuke Takagi and Kei-ichi Watanabe in \cite{TW}, who called it the $F$-pure threshold. The definition looked quite different, since they defined it using ideas from Hochster and Huneke's tight closure theory \cite{HH1}. Expanding on this idea, Hara and Yoshida \cite{HY} introduced the test ideals (under the name ``generalized test ideals" in reference to the original test ideal of Hochster and Huneke, which was not defined for pairs), and soon later using this train of thought, $F$-jumping numbers were introduced in \cite{MTW}, where they are called $F$-thresholds. 
 The definition of  $F$-pure threshold, test ideals, and the higher $F$-jumping numbers we presented here is essentially from \cite{BMS1}, where it is also proven that this point of view is equivalent (in regular rings) to the previously defined concepts.  This point of view  removes explicit mention of tight closure, focusing instead on $R^{p^e}$-linear maps (or differential operators).
 
 We have presented the definition and only the simples properties of F-threshold and test ideals, 
and only in the simplest possible case: of one polynomial in a polynomial ring, or what amounts to a hypersurface in a smooth ambient scheme. Nor have we included any substantial applications. 
We urge the reader to investigate the survey  \cite{ST}, or others previously mentioned. In particular, test  ideals are defined not only for individual polynomials, but for any ideal, and the ambient ring need not be regular (for example, the $F$-jumping numbers are discrete and rational in greater generality \cite{BSTZ}). There  are a great number of beautiful applications to developing tools from birational geometry in characteristic $p$ using test ideals, such as Schwede's ``centers of $F$-purity" and $F$-adjunction \cite{Sch1} \cite{Sch2}; see the papers of Schwede, Tucker, Takagi, Hara, Watanabe, Yoshida, Zhang, Blickle, and others listed in the bibliographies of \cite{ST}.

  \end{Rmk}

\section{Unifying the prime characteristic and zero characteristic approaches.}

We defined the log canonical threshold for complex polynomials using integration, and
the $F$-threshold for characteristic $p$ polynomials using differential operators. However, as we have seen, in characteristic zero, our approach was equivalent to a natural approach to measuring singularities  in birational geometry. 
Since birational geometry makes sense over any field, might we also be able to define the $F$-threshold directly in this world as well?

This approach does not work as well as we would hope in characteristic $p$. Two immediate problems come to mind. First, resolution of singularities is not known in characteristic $p$.  It turns out that  this is not a very serious problem.
Second, and more fatally, some of the vanishing theorems for cohomology that make multiplier ideals such a useful tool in characteristic zero actually {\it fail\/} in prime characteristic.  

The  lack of   Hironaka's Theorem  in characteristic $p$  can be circumvented as follows. We  look at  {\it all proper birational maps} $X \overset{\pi}\longrightarrow \C^N$ with $X$ normal.
 If $X$ is a normal variety, the needed machinery of divisors goes through as in the smooth case, because the singular locus of a normal variety is of codimension two or higher. 
 To define the order of vanishing of a function along an irreducible divisor $D$, we  restrict to any (sufficiently small) smooth open set meeting the divisor. Thus, the relative canonical divisor $K_{\pi}$ can be defined for a map $X \overset{\pi}\longrightarrow \C^N$, for any {\it normal\/}  $X$, as can the divisor $F = $ div$(f \circ \pi)$. 
 That is, if  $X \overset{\pi}\longrightarrow \C^N$ with $X$ normal, we can define $K_{\pi}$ and $F$ as the divisor of the Jacobian determinant of $\pi_{|U} $ and the divisor of $f\circ \pi_{|U}$ respectively, where $U \subset X$ is the smooth locus of $X$ (or any smooth subset of $X$ whose complement is codimension two or more).
 
 So we can attempt to define the log canonical threshold in arbitrary characteristic as follows:
 
\begin{defi}\label{lctp} The \emph{log canonical threshold} of  a polynomial $f \in k[x_1, \dots, x_n]$  is defined as
$$ lct(f):=\sup \left\{ \lambda  \in \R_{+} \ \big|\  \lceil K_{\pi} - \lambda F\rceil\,\,\,\, {\rm{is \,\,\,\, effective\,}} \right\},$$
as we range over all proper birational morphisms $X \overset\pi\longrightarrow \A_k^n$ with $X$ normal.
\end{defi}

\begin{Rmk} It is not hard to show that, in characteristic zero,  this definition produces the same value as Definition \ref{lct}. See  \cite[Thm 9.2.18]{Laz}. \end{Rmk}

Similarly, for any divisor on a normal variety $X$, the sheaves $\mc{O}_X(D)$ are defined.{\footnote{Although unlike the smooth case, they need not be invertible sheaves in general.}} So we can also attempt to define the multiplier ideal in characteristic $p$ similarly, by considering {\it all} proper birational models:

\begin{defi}\label{numerop}
Let $f$ be a polynomial in $n$ variables, and $\lambda$ a positive real number.  The multiplier ideal  is
$$\J( f^{\lambda}) =  \!\!\!\! \bigcap_{X \overset{\pi}\longrightarrow \A_k^n}\!\!\!\! \pi_*\mc{O}_X(\lceil K_\pi - \lambda F \rceil)  = \{ h \in \C[x_1, \ldots x_n]\  |\   {\rm div}(h) + \lceil K_\pi - \lambda F \rceil \geq 0\}
$$ 
as we range over normal varieties $X$, mapping  properly and birationally to $\A_k^n$ via $\pi$, where $K_{\pi}$ is the relative canonical divisor, and $F$ is the divisor div$(f\circ \pi)$ on $X$. 
Equivalently, this amounts to
$$
\J( f^{\lambda}) = \{h \in k[x_1, \dots, x_n]\,\, | \,\,  {\rm ord}_E(h) \geq \lambda {\rm ord}_E(f\circ \pi) - \ord_E(\Jac_{\C}(\pi))\},
$$
where we range over all irreducible divisors $E$ lying on a normal $X$ mapping properly and birationally to $\A_k^n$, say $X \overset{\pi}\longrightarrow\A_k^n$.
\end{defi}

 Again, in  characteristic 0, this produces the same definition as before.
  Does the multiplier ideal in  characteristic $p$ (produced by Definition \ref{numerop}) have the same good properties as in characteristic zero? The answer is NO. 
 The problem occurs with the behavior of multiplier ideals in prime characteristic under wildly ramified maps: they simply do not have the properties we expect of multiplier ideals based on their behavior in characteristic zero (see \cite[Example 6.33]{ST} or \cite[Example 7.12]{ST}). 
 The test ideals have better properties in char $p$ than the multiplier ideals.  They accomplish much of what multiplier ideals do in characteristic zero. 
The survey  \cite{ST} gives an excellent introduction to this topic.

One reason the multiplier ideals fail to be useful in prime characteristic is that certain vanishing theorems fail that 
contribute to the magical properties of multiplier ideals over $\mathbb C$.
 For example,  a very useful statement is ``local vanishing'': 
  If $\pi: X \rightarrow \A_k^N$ is a log resolution of a complex polynomial $f$, then
$$ R^i \pi_* \mc{O}_X(K_\pi - \lfloor c F \rfloor)=0$$
for all $i >0$  (cf. {\cite[Theorem 9.4.1]{Laz}}).
For example, local vanishing is needed  to prove the Brian\c con-Skoda theorem for non-principal ideals in characteristic zero.  Unfortunately, this vanishing theorem is false in characteristic $p$. Fortunately,  the Brian\c con-Skoda theorem for test ideals can be proven quite simply  in characteristic $p$, using Frobenius instead of vanishing theorems.

On the other hand, for ``large $p$," it is true that the multiplier ideals ``reduce mod $p$" to the test ideals.

\subsection{\bf Idea of reduction modulo $p$.}
Fix a polynomial $f\in\Q[x_1,\dots,x_n]$, or (by clearing denominators) in $\Z[x_1,\dots,x_n]$. Fix a  log resolution of $f$ over $\Q$, say given by $X_\Q \overset{\pi}{\longrightarrow}\A_\Q^n$. We can ``thicken" $X_{\Q}$ to a scheme $X_{\Z}$ over $\Z$, and so get a  family of maps over $\Spec \Z$ described by the 
 following diagram:
$$\xymatrix@C=3pc@R=3pc{
X_\Q\ar[d]_{\pi} & X_\Z\ar[d] & X_{\F_p}\ar@{_(->}[l]\ar[d]\\
\A_\Q^n\ar@{^(->}[r]\ar[d] & \A_\Z^n \ar[d] & \A^n_{\F_p}\ar@{_(->}[l]\ar[d]\\
\Spec\Q\ar@{^(->}[r] & \Spec{\Z} & \Spec(\F_p)\ar@{_(->}[l]
}$$
where the right hand side gives a fiber over a closed point $p$ in $\Spec \Z$ and the left hand side shows the generic fiber. Because the generic fiber is a log resolution of $f$, it follows that for an open set of closed fibers, we also have a log resolution of $f$. That is, 
we can assume $X_{\F_p} \longrightarrow \A_{\F_p}^n $ is a log-resolution of $f$ for $p\gg 0$.

The multiplier ideal $\J(\A^n_\Q,f^c)\subset \Q[x_1,\dots,x_n]$ can be viewed as an ideal in  $\Z[x_1, \dots, x_n]$ by clearing denominators if necessary; abusing notation we denote the ideal in $\Z[x_, \dots, x_n]$ and $\Q[x_1, \dots, x_n]$ the same way. So we 
 can reduce modulo $p$ and obtain an analog of multiplier ideals in positive characteristic given by:
$$\J(\A^n_{\F_p},f^c)=\J(\A^n_\Q,f^c)\otimes \F_{p}.$$
These turn out to be the test ideals for $p\gg 0$!

\begin{Thm}{\rm(}\cite[Theorem 3.1]{Sm00}, \cite{Hara}, \cite[Theorem 6.8]{HY}{\rm)}.

\begin{itemize}
\item For all $p\gg0$ and for all $c$,
$$\tau(\F_p[x_1,\dots,x_n],f^c)\subseteq \J(\Q[x_1,\dots,x_n],f^c)\otimes \F_p.$$
\item Fix $c$, for all $p\gg0$,
$$\tau(\F_p[x_1,\dots,x_n],f^c) =  \J(\Q[x_1,\dots,x_n],f^c)\otimes \F_p.$$  (How large is large enough for $p$ depends on $c$.)
\end{itemize}
\end{Thm}

As we have explained however, the test ideals are probably the ``right" objects to use in each particular  characteristic $p$. 

Recently, Blickle, Schwede and Tucker have found an interesting way to unify test ideals and multiplier ideals (see \cite{BST}). The idea is to look at  a broader class of  proper maps, not just birational ones.
Recall that  a surjective morphism of varieties $X \overset{\pi}\rightarrow Y$ is an {\it alteration} if it is proper and generically finite. We say an alteration is separable if
the corresponding   extension of function fields $k(Y) \subset k(X)$ is separable.  
Note that such $\pi$  always  factors as $X \overset{\phi}\rightarrow \tilde Y \overset{\nu}\rightarrow Y$ where $\phi$ is proper  birational and $\nu$ is finite.  

Consider a separable alteration $X \overset{\pi}\rightarrow \A^n$, with $X$ normal.
Denote by $F_{\pi}$ the divisor on $X$ defined by $f \circ \pi$ and by $K_{\pi}$ the divisor on $X$ defined by the Jacobian.{\footnote{By which we mean the unique divisor on $X$ which agrees with these divisors on the smooth locus of $X$. This is possible since $X$ is normal; see the beginning paragraphs of Section 4.}}
As before in our computation of the multiplier ideal, the idea is to push down the sheaf of ideals  $\mathcal O_X(\lceil K_{\pi} - \lambda F_{\pi}\rceil)$ to $\mathcal O_X$. However, this will only produce a subsheaf of $\pi_* \mathcal O_X$, which is not $\mathcal O_{\A^n} $ but rather some normal finite extension. Let us denote its global sections by $S$, which is a normal finite extension of the polynomial ring $R$.
To produce an ideal in $R$, we can use the {\it trace map.}

\subsection{Trace} Let $R \subset S$ be a finite extension of normal domains, with corresponding  fraction field extension $K \subset L$. The field trace is a $K$-linear map
$$L \rightarrow K$$
sending each $\ell \in L$  to the trace of the $K$-linear map $L \rightarrow L$ given by multiplication by $\ell.$
Because $S$ is integral over the normal ring $R$, it is easy to check that this restricts to an $R$-linear map
$$
S \overset{tr}\rightarrow R.
$$
In particular, every ideal of $S$ is sent, under the trace map, to an ideal in $R$. Using this, we can give a uniform definition of the multiplier ideal and test ideal. For any separable alteration $X \overset{\pi}\rightarrow \A^n$ with $X$ normal, denote by $tr_{\pi}$ the trace of the ring extension map $R \hookrightarrow S = \mathcal O_X(X)$.
Then we have:

\begin{Thm}{\rm(}\cite{BST}{\rm).}
Fix a polynomial  $f \in k[x_1, \dots, x_n]$ where $k$ is an arbitrary field,  and let $c$ be any positive real number. Define  
$$J:=\bigcap_\pi tr_{\pi}(\pi_*\calo_X(\lceil K_{\pi} -  cF_{\pi}\rceil)). $$ where $\pi$ varies over all possible normal varieties $X$ mapping properly and generically separably to $\A^n$. Then,
$$
J=\left\{
\begin{array}{ll}
\J(f^c) & \hbox{if the characteristic is }0,\\
\tau(f_p^c) & \hbox{if the characteristic is }p>0.
\end{array}\right.
$$
\end{Thm}

Note that each $ \pi_*\calo_X(\lceil K_{\pi} -  c\pi^*D\rceil) $ is an ideal in $\pi_*\calo_X$,  whose global sections  form some finite extension $S$ of the polynomial ring $R = k[x_1, \dots, x_n].$ So its image under the trace map is an ideal in $R$.
 The theorem says that if we intersect all such ideals of $R$, we get the test/multiplier ideal of $f^{\lambda}$.

In fact, Blickle, Schwede and Tucker prove even more: The intersection stabilizes. So there is {\it one\/} alteration $X \overset{\pi}\rightarrow \A^n$ for which 
$\tau(f^c) = 
tr_{\pi}(\pi_*\calo_X(\lceil K_{\pi} -  c\pi^*D\rceil)).$ 
In fact,  it has been shown that, fixing $f$, there is one alteration which computes all  test ideals $\tau(f^{\lambda})$, for any $\lambda$ \cite{SchTZ}. 
Of course, this is already known for multiplier ideals: it suffices to take one log resolution of $X$ to compute the multiplier ideal. Indeed, in characteristic zero, one need not take any finite covers at all. Interestingly, in characteristic $p$, it is precisely the finite covers that matter most.

It is worth remarking that many of the features of multiplier ideals---including the important local vanishing theorem---can be shown in characteristic $p$ ``up to finite cover" \cite[Theorem 5.5]{BST}. This can be viewed as a generalization  of Hochster and Huneke's original Theorem on the Cohen-Macaulayness of the absolute integral closure of a domain of characteristic $p$ \cite{HH3}.  See also \cite{SmSantaCruz}.

\end{document}